\documentclass[12pt]{article}

\usepackage{graphicx,amsmath,amsfonts,amssymb}
\usepackage[numbers]{natbib}
\usepackage[dvips]{color}
\usepackage{empheq}
\usepackage{cases}

\usepackage[dvipsnames]{xcolor}
\usepackage{here}
\usepackage{bm}
\usepackage{comment}
\usepackage{enumerate}
\usepackage{ulem}
\usepackage{ifthen}

\usepackage{enumerate}

\definecolor{brown}{cmyk}{0,0.83,1,0.6}

\newcommand{\rmn}[1]{\if#11I\else {\if#12I\hspace{-0.12ex}I\hspace{-0.08ex}\else {\if #13I\hspace{-0.12ex}I\hspace{-0.12ex}I\hspace{-1.5ex} \else{\if#14I\hspace{-0.16ex}V\hspace{-2.4ex} \else{\if#15V\hspace{-3.0ex} \else{\if#16V\hspace{-0.12ex}I\hspace{-3.0ex} \else{\if#16V\hspace{-0.12ex}I  \else{\if#17V\hspace{-0.12ex}I\hspace{-0.12ex}I\hspace{-3.0ex}  \else{\if#18V\hspace{-0.12ex}I\hspace{-0.12ex}I\hspace{-0.12ex}I\hspace{-3.8ex} \else{\if#19I\hspace{-0.12ex}X\hspace{-3.8ex} \fi}\fi} \fi}\fi}\fi}\fi} \fi} \fi} \fi}\fi}

 \newcommand{\sr}[1]{{\mathcal #1}}

 \newcommand{\eqn}[1]{(\ref{eqn:#1})}
 \newcommand{\lem}[1]{Lemma~\ref{lem:#1}}
 \newcommand{\cor}[1]{Corollary~\ref{cor:#1}}
 \newcommand{\thr}[1]{Theorem~\ref{thr:#1}}

 \newcommand{\rem}[1]{Remark~\ref{rem:#1}}
 \newcommand{\fig}[1]{Figure~\ref{fig:#1}}
\newcommand{\app}[1]{Appendix~\ref{app:#1}}

\newcommand{\enu}[1]{(\ref{enu:#1})}

\newcommand{\figt}[1]{\ref{fig:#1}}

 \newcommand{\pend}{\hfill \thicklines \framebox(6.6,6.6)[l]{}}
 \newenvironment{proof}{\noindent {\sc  Proof.} \rm}{\pend}
 \newenvironment{proof*}[1]{\noindent {\sc  #1.} \rm}{\pend}

\usepackage{theorem}
\theorembodyfont{\normalfont}
 \newtheorem{theorem}{Theorem}[section]
 \newtheorem{lemma}{Lemma}[section]

 \newtheorem{remark}{Remark}[section]

 \newtheorem{corollary}{Corollary}[section]

 \newcommand{\setnewcounter} {
 \setcounter{subsection}{0}
 \setcounter{equation}{0}
 \setcounter{conjecture}{0}
 \setcounter{assumption}{0}
 \setcounter{question}{0}
 \setcounter{definition}{0}
 \setcounter{theorem}{0}
 \setcounter{corollary}{0}
 \setcounter{lemma}{0}
 \setcounter{proposition}{0}
 \setcounter{remark}{0}
}

 \title{\bf Diffusion limit for the stationary distribution of a history-dependent two-level M/M/$1$ queue}
\author{
Masahiro Kobayashi${}^{1}$, Masakiyo Miyazawa${}^{2}$, Yutaka Sakuma${}^{3}$
\\ 
{\small ${}^{1}$ Tokai University,
 \quad ${}^{2}$ Tokyo University of Science, \quad ${}^{3}$ National Defense Academy of Japan
}
\\
{\small Email: 21506@ms.dendai.ac.jp}
}

\date{\today \,\,final ver}

\usepackage{mathrsfs}

\graphicspath{{input/Fig/}}

\begin{document}
\maketitle
\begin{abstract}
Recently, \citet{AtarMiya2025} introduced a multi-level GI/G/1 queue with a finite number of levels, where both the arrival and service rates depend on the level corresponding to the current queue length. For this model, they proved that the diffusion limit of its queue length process in heavy traffic is the level-dependent reflected Brownian motion of \cite{Miya2024b}. In a subsequent study, \citet{KobaMiyaSaku2025} derived the corresponding diffusion limit of the stationary distribution. These studies are motivated by the control of service capacity depending on the queue length. We are interested in the more general case where this control may also depend on the history of the queue length. As the first step toward such a generalization, we specialize the multi-level GI/G/$1$ queue to a two-level M/M/$1$ queue. We then extend the dynamics of this model so that its arrival and service rates depend not only on the current queue length but also on the recent history of queue lengths. Under the stability condition for this model, we first compute its stationary distribution in closed form, then derive its diffusion limit in heavy traffic. Finally, using this diffusion limit, we derive approximation formulas for the stationary distribution and then numerically assess their accuracy.
\end{abstract}

\section{Introduction}
\label{sec:Introduction}
In recent years, dynamic mechanisms to adjust processing capacity and resource allocation based on the system state have become increasingly important in various domains, such as cloud computing, communication networks, manufacturing lines, and healthcare services. For instance, many real-world systems implement controls that increase the processing speed or the number of servers as the queue length grows.
It has also been widely observed that arrival rates may vary with the system state. In web services and medical appointment systems, where congestion levels are visible, users may avoid congested periods or choose less congested time slots, thereby exhibiting self-regulation. In such situations, the effective arrival rate depends on the queue length, making the arrival process state-dependent.

To address such practical needs, a body of research exists on queueing models with adjustable processing capacity that depends on congestion. For instance, Altalejo et al. \cite{ArtaEconLope2005}, motivated by server operations in data centers, analyzed an M/M/c-type model with dynamically activated and deactivated servers, utilizing a quasi-birth-and-death process structure that limits the number of simultaneously active servers. Gandhi et al. \cite{GandBaltAdan2010} and Phung-Duc and Kawanishi \cite{TuanKawa2020} introduced a multi-server model that accounts for setup delays due to server switching, and computed performance measures. Schwartz et al. \cite{schw} proposed a multi-server setup queueing model with busy, idle, and off states for servers.
Their work was motivated by energy-saving and by the dependencies among transitions in queue length and the three states of servers. Building on that work, Mitrani \cite{Mitr2013} developed a more general model that incorporates both history dependence and non-negligible server setup times, providing an exact analytical solution for this more complex scenario. While these studies provide performance metrics for the models, understanding the relationships between these metrics (e.g., mean queue length, waiting time) and the fundamental parameters (e.g., arrival and service rates) remains challenging.

One powerful technique for tackling this challenge is heavy traffic analysis, which approximates a queueing process by a diffusion process under conditions where the system is heavily loaded. This approach often reveals the fundamental relationships between system parameters and performance metrics in a more tractable form than exact analysis. In a related line of research, queueing models where arrival and service rates depend on the queue length have attracted significant attention. To analyze their limiting behavior, continuous models such as reflected diffusion processes with state-dependent coefficients have been studied. Specifically, Atar and Miyazawa \cite{AtarMiya2025} proved that the diffusion limit of the queue length process in heavy traffic is a reflected diffusion process whose drift and diffusion coefficients depend on the level, and showed that this process is uniquely obtained as a weak solution of the associated stochastic differential equation (SDE). Furthermore, Kobayashi et al.~\cite{KobaMiyaSaku2025} derived closed-form expressions for the limit of the diffusion-scaled stationary distribution of a single-server queue with level-dependent transitions under heavy traffic conditions. In these studies, the state space of the models is essentially one-dimensional, and the transition structures do not incorporate additional information such as history or background states.

Motivated by the aforementioned studies, our study focuses on the queueing model that depends not only on the level but also on the queue-length history, and derives the diffusion approximation of the stationary distribution. We consider a queueing model in which the service rate is adjusted not only by the current queue length but also by the recent trajectory of the queue-length process. This allows us to capture more complex and realistic control mechanisms. As a first step, we introduce a two-level M/M/$1$ model where the transition rates depend on both the current congestion level and its recent history of queue length. We adopt the M/M/$1$ framework as a baseline for this initial analysis due to its fundamental importance and tractable structure. For this model, we first compute its stationary distribution in closed form. We then derive its diffusion limit in heavy traffic, which provides clearer insights into the relationship between system parameters and performance.
While an exact stationary distribution provides a complete characterization of the system's long-run behavior, its expression can be complex. The diffusion approximation, in contrast, often yields a more tractable formula that offers clearer insights into the fundamental relationships between system parameters and performance metrics, which is particularly valuable for system design and control. 

The main contributions of this study are as follows:
\begin{itemize}
  \item We construct an M/M/1 queue that incorporates history and level dependence, and derive its stationary distribution in closed form.
  \item We obtain the diffusion limit of the stationary distribution of this M/M/$1$ queue under heavy traffic conditions for each background state. Notably, some of these limiting distributions exhibit nonstandard forms that are neither exponential nor uniform, distinguishing them from results typically seen in diffusion limit studies.
  \item Using the resulting density functions, we clarify the relationships between key performance measures and the underlying model parameters.
  \item Through numerical examples, we assess the accuracy of the diffusion approximation and demonstrate its practical effectiveness.
\end{itemize}

The remainder of this paper is organized as follows. Section 2 formulates the M/M/1 model with a state- and history-dependent transition structure and derives the stationary distribution in closed form. Section 3 derives the diffusion limit of the stationary distribution under heavy traffic conditions. Section 4 provides the proofs of the main theorems presented in the preceding sections. Section 5 presents numerical examples to assess the accuracy of our diffusion approximation and validate the theoretical results. Section 6 concludes the paper and outlines directions for future research.

\section{History-dependent two-level M/M/$1$ queue}
\setnewcounter
\label{sec:model}

In this section, we formally define our queueing model. We introduce some standard notations. Let $\mathbb{R}$ and $\mathbb{Z}$ be the sets of all real numbers and all integers, respectively, and
\begin{align*}
&\mathbb{R}_{+} =\{x \in \mathbb{R} ; x \ge 0\}, && \mathbb{R}_{>0} = \{x \in \mathbb{R}; x > 0\},  && \mathbb{R}_{-} =\{x \in \mathbb{R} ; x \le 0\},\\
&\mathbb{R}_{<0} = \{x \in \mathbb{R}; x < 0\}, &&\mathbb{Z}_{+} = \{x \in \mathbb{Z} ; x \ge 0\}, &&\mathbb{N} = \{x \in \mathbb{Z}; x > 0\}.
\end{align*}

Our queueing model extends the standard M/M/1 queue by allowing the arrival and service rates to depend on the history and level of the queue length.
Let $L(t)$ be the queue length at time $t \in \mathbb{R}_{+}$ (including any customer in service), and let $\ell_{d}$ and $\ell_{u}$ be positive integers satisfying $\ell_{d} < \ell_{u}$.
We define two overlapping regions, which we refer to as levels. The set $\{0,1,\dots,\ell_{u}\}$ is called a level 1, and the set $\{\ell_{d},\ell_{d} + 1,\dots\}$ is called a level 2.
Obivously, these two levels intercect on $ \{\ell_{d},\ell_{d} + 1,\dots,\ell_{u}\}$.
When the queue length belongs to $\{\ell_{d},\ell_{d} + 1,\dots,\ell_{u}\}$, additional information is needed to determine the effective arrival and service rates.
To provide this information, we introduce a background state process $\{B(t); t \in \mathbb{R}_{+}\}$, which takes values in $\{1,2\}$.
The state of the background process changes based on the queue length process $\{L(t); t \in \mathbb{R}_{+}\}$ as follows:
\begin{itemize}
    \item If $B(t-) = 1$ and an arrival causes the queue length to up-cross the threshold $\ell_u$ (i.e., $L(t-) = \ell_{u}$ and $L(t) = \ell_{u} + 1$), then the background state switches to $B(t) = 2$.
    \item If $B(t-) = 2$ and a service completion causes the queue length to down-cross the threshold $\ell_d$ (i.e., $L(t-) = \ell_{d}$ and $L(t) = \ell_{d}-1$), then the background state switches to $B(t) = 1$.
\end{itemize}
For each background state $i \in \{1,2\}$, we associate constants arrival and service rates, $\lambda_{i}, \mu_{i} \in \mathbb{R}_{>0}$.
The effective rates at time $t$ are then given by $\lambda_{B(t)}$ and $\mu_{B(t)}$.

These structures introduce history dependence. For instance, consider a time $t_1$ when the queue length is $L(t_1) = \ell_u$. If the system reached this state via a departure from $\ell_u+1$, the background state would be $B(t_1)=2$, making the effective rates $(\lambda_2, \mu_2)$. In contrast, if at another time $t_2$ the system reached $L(t_2) = \ell_u$ via an arrival from $\ell_u-1$, the background state would be $B(t_{2})=1$, and the rates would be $(\lambda_1, \mu_1)$. Thus, the effective rates depend not only on the current state but also on the history of the process. For this reason, we refer to our model as a history-dependent two-level M/M/$1$ queue.
Figures \figt{sample path} and \figt{transition} illustrate a sample path and the state transition diagram of the model, respectively.

\begin{figure}[h]
 	\centering
	\includegraphics[bb=0 0 4000 1300, scale = 0.14]{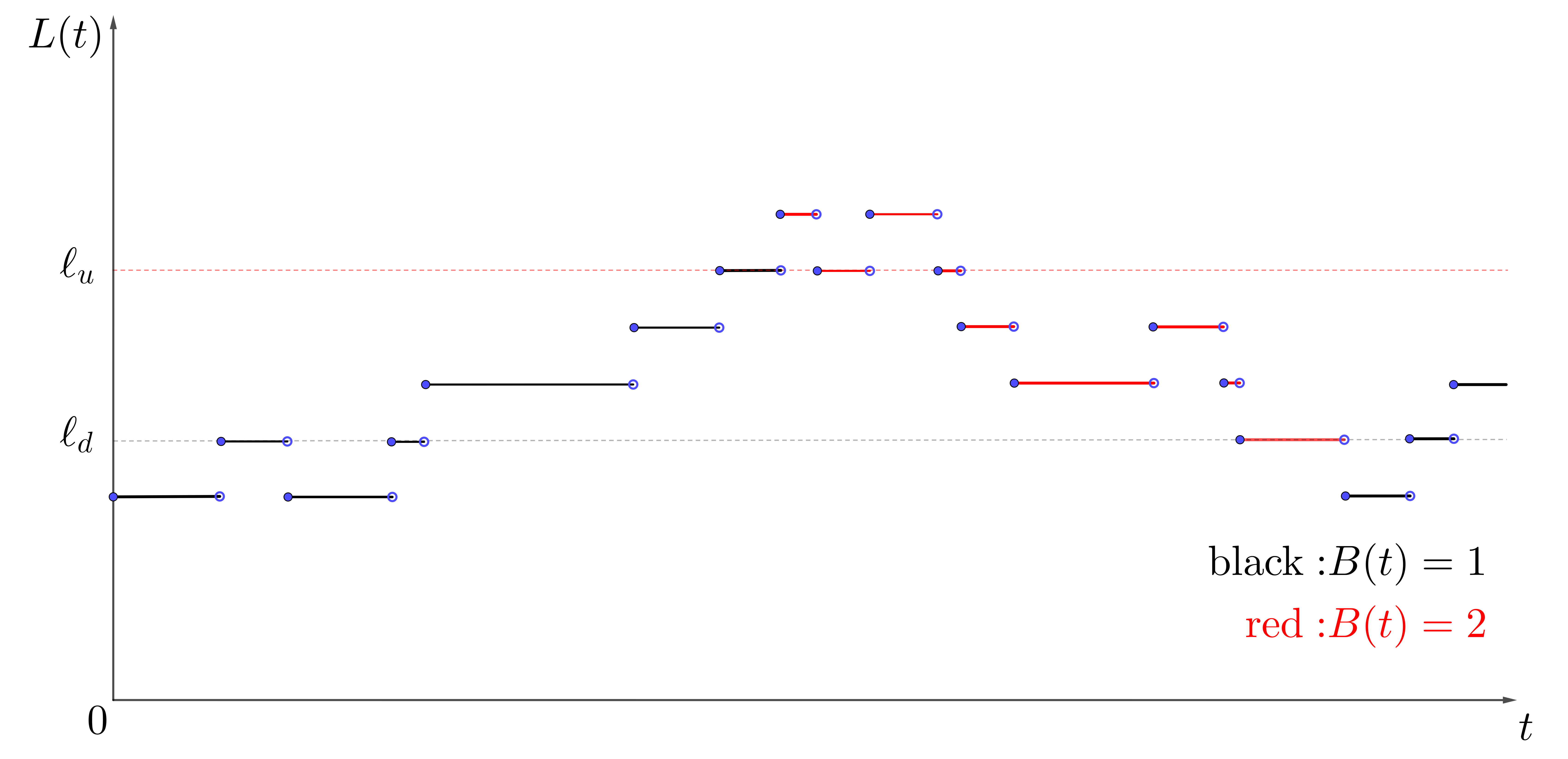}
	\caption{Sample path of $L(t)$}
	\label{fig:sample path}
\end{figure}

\begin{figure}[h]
 	\centering
	\includegraphics[height=0.21\textheight]{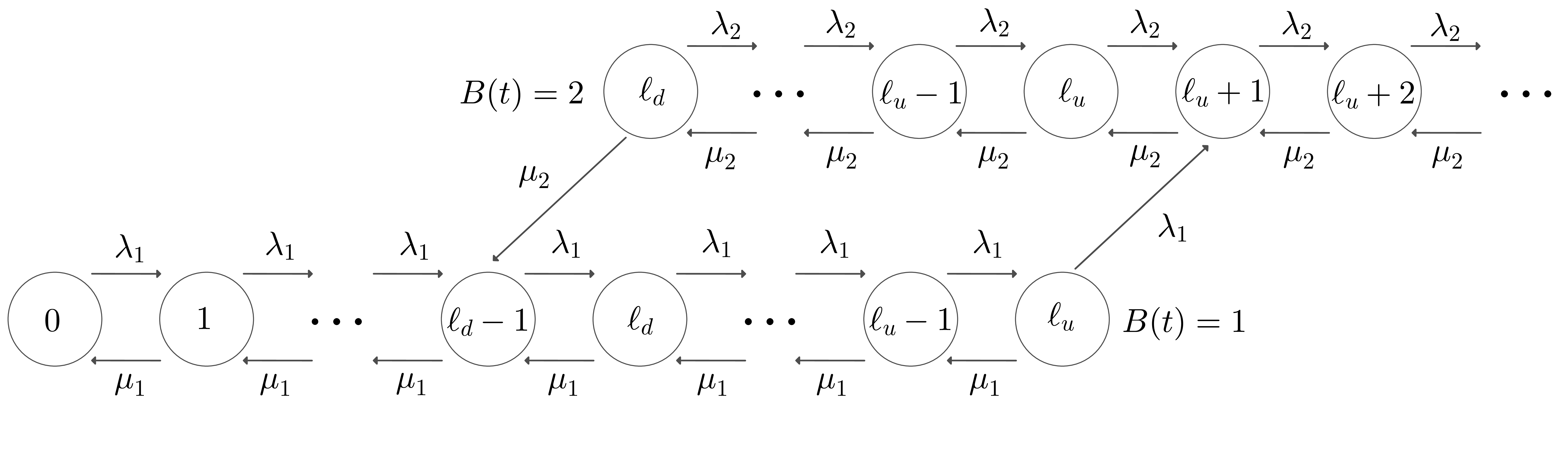}
	\caption{Transition diagram of $(L(t),B(t))$}
	\label{fig:transition}
\end{figure}

We now formulate the joint process $\{(L(t), B(t)); t \in \mathbb{R}_{+}\}$ as a continuous-time Markov chain and proceed to derive its stationary distribution. 
The state space of this process is denoted by $S = S_{1} \cup S_{2}$, where $S_{1}  = \{0,1,\dots,\ell_{u}\} \times \{1\}$ and $S_{2}  =  \{\ell_{d},\ell_{d} + 1, \dots \}\times \{2\}$.
Let $Q = \{q_{(\ell,k), (\ell',k')}; (\ell,k),(\ell',k') \in S\}$ be a square matrix whose off-diagonal elements are given by, for $(\ell,k) \neq (\ell',k')$,
\begin{align*}
q_{(\ell,k), (\ell',k')} = 
\begin{cases}
\lambda_{i}, & (\ell,k),(\ell',k')\in S_{i}, \ell' = \ell + 1, i=1,2,\\ 
\mu_{i}, & (\ell,k),(\ell',k')\in S_{i}, \ell' = \ell - 1, i=1,2,\\
\lambda_{1}, & (\ell,k) \in S_{1}, (\ell',k') \in S_{2}, \ell = \ell_{u},  \ell ' = \ell_{u} + 1,\\
\mu_{2}, & (\ell,k) \in S_{2}, (\ell',k') \in S_{1},\ell = \ell_{d}, \ell' = \ell_{d} - 1,\\
0, & \text{otherwise,}
\end{cases}
\end{align*}
and the diagonal elements are given by $q_{(\ell,k), (\ell,k)} = -\sum_{(\ell',k') \neq(\ell,k)} q_{(\ell,k), (\ell',k')}$.
Then, the matrix $Q$ is the transition rate matrix of $\{(L(t),B(t))\}$.

Let $\rho_{i} = \dfrac{\lambda_{i}}{\mu_{i}}$ for $i= 1,2$ and $\rho_{12} = \dfrac{\lambda_{1}}{\mu_{2}}$. If a stationary distribution for the process $\{(L(t),B(t))\}$ exists, we denote it by $\pi = (\pi(\ell,k); (\ell,k) \in S)$.
It is easy to see that the stationary distribution $\pi$ exists if and only if 
\begin{align}
\label{eqn:stability}
\rho_{2} < 1.
\end{align}
Notably, we do not require the condition $\rho_{1} < 1$. 
Throughout this paper, we assume the stability condition \eqn{stability}.
 We define a set $\partial S$ as
\begin{align*}
\partial S = \{(0,1), (\ell_{d}-1,1), (\ell_{d},2), (\ell_{u},1),(\ell_{u}+1,2) \}.
\end{align*}
The stationary distribution $\pi$ then satisfies the following balance equations:
\begin{align}
\label{eqn:sta1}
&(\lambda_{k} + \mu_{k}) \pi(\ell,k) =  \lambda_{k} \pi(\ell - 1,k) + \mu_{k}\pi(\ell+1,k), \quad (\ell,k) \in S \setminus \partial S,\\
\label{eqn:sta2}
&\lambda_{1} \pi(0,1) =  \mu_{1} \pi(1,1),\\
\label{eqn:sta3}
&(\lambda_{1} + \mu_{1}) \pi(\ell_{d}-1,1) =  \lambda_{1}\pi(\ell_{d}- 2,1) + \mu_{1} \pi(\ell_{d},1) + \mu_{2} \pi(\ell_{d},2),\\
\label{eqn:sta4}
&(\lambda_{2} + \mu_{2}) \pi(\ell_{d},2) =  \mu_{2} \pi(\ell_{d}+1,2), \\
\label{eqn:sta5}
&(\lambda_{1} + \mu_{1}) \pi(\ell_{u},1) =  \lambda_{1} \pi(\ell_{u}- 1,1),\\
\label{eqn:sta6}
&(\lambda_{2} + \mu_{2}) \pi(\ell_{u} + 1,2) =  \lambda_{2} \pi(\ell_{u},2) + \mu_{2} \pi(\ell_{u}+2,2) + \lambda_{1} \pi(\ell_{u},1).
\end{align}

To facilitate the analysis, we partition the state space $S$ into four subsets, $S_{1,1},S_{1,2},S_{2,1},S_{2,2}$, defined by
\begin{align*}
&S_{1,1} = \{0,1,\dots,\ell_{d}-1\} \times \{1\}, && S_{2,1} = \{\ell_{d},\ell_{d}+1,\dots,\ell_{u}\} \times \{1\},\\
&S_{1,2} = \{\ell_{d},\ell_{d}+1,\dots,\ell_{u}\} \times \{2\}, 
&&S_{2,2} = \{\ell_{u}+1,\ell_{u} + 2, \dots \} \times \{2\}.
\end{align*}
For a condition {\rm A}, let $1(\mathrm{A})$ be 
\begin{align*}
1({\rm A}) = 
\begin{cases}
1, & \mbox{if A is true}, \\
0, & \mbox{if A is false}.
\end{cases}
\end{align*}
For convenience, let $\phi$ be the auxiliary function on $\mathbb{Z}_{+}$ defined by
\begin{align}
\label{eqn:phi k}
\phi(i) = \sum_{j = 0}^{i} \rho_{1}^{j} = \frac{1 - \rho_{1}^{i + 1}}{1 - \rho_{1}}1(\rho_{1} \neq 1) + (i+1) 1(\rho_{1} = 1), \qquad i \in \mathbb{Z}_{+}.
\end{align}
Then, solving the balance equations \eqn{sta1}--\eqn{sta6} yields the following theorem.

\begin{theorem}
\label{thr:sta exact}
Under the stability condition \eqn{stability}, the stationary distribution $\pi$ is given by the following expressions:
\begin{align}
\label{eqn:sta exact1}
\pi(\ell,k) =  
\begin{cases}
\rho_{1}^{\ell} \pi(0,1),&(\ell,k) \in S_{1,1},\\
\dfrac{\phi(\ell_{u}) - \phi(\ell-1)}{\phi(\ell_{u}- \ell_{d} + 1)} \pi(0,1), &(\ell,k) \in S_{2,1},\\
\dfrac{\rho_{1}^{\ell_{u}}}{\phi(\ell_{u} - \ell_{d} + 1)} \dfrac{1 - \rho_{2}^{\ell - (\ell_{d}-1)}}{1 - \rho_{2}} \rho_{12} \pi(0,1), &(\ell,k) \in S_{1,2},\\
\dfrac{\rho_{1}^{\ell_{u}}}{\phi(\ell_{u}- \ell_{d} + 1)}\dfrac{1 - \rho_{2}^{\ell_{u} - \ell_{d} + 2}}{1 - \rho_{2}} \rho_{12} \rho_{2}^{\ell - (\ell_{u} + 1)}\pi(0,1), &(\ell,k) \in S_{2,2},
\end{cases}
\end{align}
where the normalization constant $\pi(0,1)$ is given by
\begin{align}
\label{eqn:pi0 exact}
\pi(0,1) = \left(\phi(\ell_{d}-1) + \sum_{\ell=\ell_{d}}^{\ell_{u}}\dfrac{\phi(\ell_{u}) - \phi(\ell-1)}{\phi(\ell_{u}- \ell_{d} + 1)}  +  \dfrac{\rho_{12}\rho_{1}^{\ell_{u}}\left(\ell_{u} - \ell_{d} + 2 \right)}{\phi(\ell_{u} - \ell_{d} + 1)(1 - \rho_{2})}  \right)^{-1}.
\end{align}
\end{theorem}

\begin{remark}
\label{rem:closed form}
It follows from \eqn{phi k} that \eqn{sta exact1} can be  written for $(\ell,k) \in S_{2,1}$ as
\begin{align*}
&\pi(\ell,k) = \left(\dfrac{\rho_{1}^{\ell} - \rho_{1}^{\ell_{u}+1}}{1 - \rho_{1}^{\ell_{u}  - \ell_{d} + 2}}1(\rho_{1} \neq 1) 
 + \dfrac{(\ell_{u}+1) - \ell}{\ell_{u} - \ell_{d}+2}1(\rho_{1} = 1)\right)\pi(0,1),
\end{align*}
for $(\ell,k) \in S_{1,2}$ as
\begin{align*}
\pi(\ell,k) = 
\left(\dfrac{\rho_{1}^{\ell_{u}} - \rho_{1}^{\ell_{u}+1}}{1 - \rho_{1}^{\ell_{u} - \ell_{d} + 2}}1(\rho_{1} \neq 1) +  \dfrac{1}{\ell_{u} - \ell_{d} + 2}1(\rho_{1} = 1)\right)\dfrac{1 - \rho_{2}^{\ell - (\ell_{d}-1)}}{1 - \rho_{2}} \rho_{12} \pi(0,1),
\end{align*}
and for $(\ell,k) \in S_{2,2}$ as
\begin{align*}
\pi(\ell,k) = \left(\dfrac{\rho_{1}^{\ell_{u}} - \rho_{1}^{\ell_{u}+1}}{1 - \rho_{1}^{\ell_{u} - \ell_{d} + 2}}1(\rho_{1} \neq 1) +  \dfrac{1}{\ell_{u}- \ell_{d} + 2}1(\rho_{1} = 1)\right)\dfrac{1 - \rho_{2}^{\ell_{u} - \ell_{d}}}{1 - \rho_{2}} \rho_{12} \rho_{2}^{\ell - \ell_{u} + 1} \pi(0,1).
\end{align*}
Moreover, from \eqn{pi0 exact}, the normalization constant can be expressed as
\begin{align*}
\pi(0,1) = 
\begin{cases}
\left(\dfrac{1}{1-\rho_{1}} - \dfrac{\rho_{1}^{\ell_{u} + 1}(\ell_{u} - \ell_{d} +2)}{(1 - \rho_{1}^{\ell_{u} - \ell_{d} + 2})} + \dfrac{\rho_{12}(\rho_{1}^{\ell_{u}} - \rho_{1}^{\ell_{u} + 1})(\ell_{u}  - \ell_{d} + 2) }{(1-\rho_{1}^{\ell_{u} - \ell_{d}  + 2})(1 - \rho_{2})} \right)^{-1}, & \rho_{1} \neq 1,\\
\left(\dfrac{\ell_{u} + \ell_{d} + 1}{2} + \dfrac{\rho_{12}}{1 - \rho_{2}} \right)^{-1}, & \rho_{1} = 1.
\end{cases}
\end{align*}
\end{remark}

The proof of \thr{sta exact} is provided in \app{pi}. 
Let $(L,B)$ be a random vector subject to the stationary distribution $\pi$. We also denote the moment generating function (MGF) of $L$ by $\psi$, defined as
\begin{align*}
\psi(\theta) = \mathbb{E}\left(e^{\theta L} \right), \qquad \theta \in \mathbb{R}.
\end{align*}
We note that $\psi(\theta) < \infty$ for all $\theta \le 0$. 
Furthermore, for $i,j = 1,2$, let
\begin{align*}
\psi_{i,j}(\theta) = \mathbb{E} \left(e^{\theta L}1((L,B) \in S_{i,j}) \right) = \sum_{(\ell,k) \in S_{i,j}} e^{\theta \ell} \pi(\ell,k), \qquad \theta \in \mathbb{R}.
\end{align*}
Then, $\psi_{1,1}(\theta)$, $\psi_{2,1}(\theta)$, and $\psi_{1,2}(\theta)$ are finite for all $\theta \in \mathbb{R}$. In contrast, $\psi_{2,2}(\theta)$ is finite only if $\theta < \log \rho_{2}^{-1}$.
From \rem{closed form}, for $\rho_{1} \neq 1$, we have
\begin{align}
\label{eqn:psi 11}
&\psi_{1,1}(\theta) = \frac{1 - (\rho_{1}e^{\theta})^{\ell_{d}}}{1 - \rho_{1}e^{\theta}} \pi(0,1),\\
\label{eqn:psi 21}
&\psi_{2,1}(\theta) = \frac{\frac{(\rho_{1}e^{\theta})^{\ell_{d}} - (\rho_{1} e^{\theta})^{\ell_{u} + 1}}{1 - \rho_{1}e^{\theta}} - \rho_{1}^{\ell_{u} + 1}\frac{e^{\theta \ell_{d}} - e^{\theta(\ell_{u} + 1)}}{1 - e^{\theta}}}{1 - \rho_{1}^{\ell_{u}-\ell_{d} + 2}}\pi(0,1), \qquad \theta \neq 0,\\
\label{eqn:psi 12}
&\psi_{1,2}(\theta) = \dfrac{\rho_{12}(\rho_{1}^{\ell_{u}} - \rho_{1}^{\ell_{u}+1})}{(1 - \rho_{1}^{\ell_{u} - \ell_{d} + 2})(1- \rho_{2})}  \left(\frac{e^{\theta \ell_{d}} - e^{\theta (\ell_{u} + 1)}}{1 - e^{\theta}} - e^{\theta(\ell_{d}-1)} \frac{\rho_{2}e^{\theta} - (\rho_{2} e^{\theta})^{\ell_{u} - \ell_{d} + 2}}{1 - \rho_{2}e^{\theta}} \right)\pi(0,1),  \nonumber\\
& \qquad \qquad \qquad \qquad \quad \qquad \qquad \qquad \qquad \qquad \qquad \qquad \qquad \qquad \qquad \theta \neq 0, \log \rho_{2}^{-1},\\
\label{eqn:psi 22}
&\psi_{2,2}(\theta) = \dfrac{\rho_{12}(\rho_{1}^{\ell_{u}} - \rho_{1}^{\ell_{u}+1}) (1-\rho_{2}^{\ell_{u} - \ell_{d} + 2})}{(1 - \rho_{1}^{\ell_{u} - \ell_{d} + 2})(1- \rho_{2})}  \frac{e^{\theta (\ell_{u} + 1)}}{1 - \rho_{2}e^{\theta} }\pi(0,1),\qquad \theta < \log \rho_{2}^{-1}.
\end{align}
For the case $\rho_1 = 1$, the expressions are as follows:
\begin{align}
\label{eqn:psi 11 1}
&\psi_{1,1}(\theta) = \frac{1 - e^{\theta \ell_{d}}}{1 - e^{\theta}} \pi(0,1), \qquad  \theta \neq 0,\\
\label{eqn:psi 21 1}
&\psi_{2,1}(\theta) = \frac{(\ell_{u} - \ell_{d} + 1)e^{\theta \ell_{d}} -\frac{e^{\theta(\ell_{d} + 1)} - e^{\theta(\ell_{u} + 2)}}{1 - e^{\theta}}}{(\ell_{u} - \ell_{d} + 2)(1 -e^{\theta})}\pi(0,1), \qquad  \theta \neq 0,\\
\label{eqn:psi 12 1}
&\psi_{1,2}(\theta) = \dfrac{\rho_{12}}{(\ell_{u}-\ell_{d} + 2)(1-\rho_{2})}  \left(\frac{e^{\theta \ell_{d}} - e^{\theta (\ell_{u} + 1)}}{1 - e^{\theta}} - e^{\theta(\ell_{d}-1)} \frac{\rho_{2}e^{\theta} - (\rho_{2} e^{\theta})^{\ell_{u} - \ell_{d} + 2}}{1 - \rho_{2}e^{\theta}} \right)\pi(0,1), \nonumber \\
& \qquad \qquad \qquad \qquad \quad \qquad \qquad \qquad \qquad \qquad \qquad \qquad \qquad \qquad \qquad \theta \neq 0, \log \rho_{2}^{-1},\\
\label{eqn:psi 22 1}
&\psi_{2,2}(\theta) = \dfrac{\rho_{12}(1-\rho_{2}^{\ell_{u} - \ell_{d} + 2})}{(\ell_{u}-\ell_{d} + 2)(1-\rho_{2})}  \frac{e^{\theta (\ell_{u} + 1)}}{1 - \rho_{2}e^{\theta} }\pi(0,1),\qquad \theta < \log \rho_{2}^{-1}.
\end{align}

\section{Main result}
\setnewcounter
\label{sec:main}
In this section, we derive a diffusion limit of the stationary distribution $\pi$.
Following the standard approach in diffusion analysis, we consider a sequence of queueing systems, indexed by $n \in \mathbb{N}$. We append the superscript $(n)$ to any notation associated with the $n$-th system, such as $\ell_{u}^{(n)}$, $\ell_{d}^{(n)}$, and $(L^{(n)},B^{(n)})$. 
Our goal is to obtain the limiting distribution of the scaled queue length $\dfrac{1}{\sqrt{n}}L^{(n)}$ as $n \to \infty$, which we refer to as the diffusion limit of the stationary distribution. To this end, we assume the following conditions. For functions $f$ and $g$, we write $g(x) = o(f(x))$ as $x \to \infty$ if 
\begin{align*}
\lim_{x \to \infty} \frac{g(x)}{f(x)} = 0.
\end{align*} 
\begin{enumerate}[(a)]
\item For all $n \in \mathbb{N}$, the $n$-th system is stable: $\rho_{2}^{(n)} < 1$.
\label{enu:stability}
\item For $i = 1,2$, there are constants $\widetilde{\lambda}_{i}, \widetilde{\mu}_{i} \in \mathbb{R}_{>0}$ such that $\lambda_{i}^{(n)} \to \widetilde{\lambda}_{i}$ and $\mu_{i}^{(n)} \to \widetilde{\mu}_{i}$ as $n \to \infty$.
\label{enu:basic}
\item There are positive constants $\widetilde{\ell}_{d},\widetilde{\ell}_{u} \in \mathbb{R}_{>0}$ such that $\widetilde{\ell}_{d} < \widetilde{\ell}_{u}$ and 
\begin{align*}
\ell_{d}^{(n)} = \sqrt{n} \widetilde{\ell}_{d} + o(\sqrt{n}), \quad \ell_{u}^{(n)} = \sqrt{n}\widetilde{\ell}_{u} + o(\sqrt{n}),
\end{align*}
 as $n \to \infty$. 
\label{enu:ell con}
\item There are constants $b_{1} \in \mathbb{R}$ and $b_{2} \in \mathbb{R}_{<0}$ such that
\begin{align}
\label{eqn:scale rho}
\rho_{i}^{(n)} = 1 + \dfrac{b_{i}}{\sqrt{n}} + o\left(\dfrac{1}{\sqrt{n}} \right), \qquad i = 1,2,
\end{align}
as $n \to \infty$.
\label{enu:rho}
\end{enumerate}

A few remarks on these assumptions are in order.
The stability of the $n$-th system requires Assumption \enu{stability}.
Assumption \enu{basic} implies that $\rho_{12}^{(n)} = \widetilde{\rho}_{12} + o(1)$  where $\widetilde{\rho}_{12} = \dfrac{\widetilde{\lambda}_{1}}{\widetilde{\mu}_{2}}$. 
The conditions $\ell_{d}^{(n)} < \ell_{u}^{(n)}$ for all $n$ and the convergence in \enu{ell con} together imply $\widetilde{\ell}_{d} \le \widetilde{\ell}_{u}$. To simplify our discussion , we assume $\widetilde{\ell}_{d} < \widetilde{\ell}_{u}$. From Assumption \enu{rho}, 
\begin{align}
\label{eqn:lambda mu lim}
\lambda_{i}^{(n)} - \mu_{i}^{(n)} = \frac{b_{i} }{\sqrt{n}}\mu_{i}^{(n)} + o\left(\frac{1}{\sqrt{n}}\right), \qquad i = 1,2, \quad n \to \infty.
\end{align}
From \eqn{lambda mu lim} and Assumption \enu{basic}, we must have $\widetilde{\lambda}_{i} = \widetilde{\mu}_{i}$ for $i=1,2$.
Assumptions \enu{ell con} and \enu{rho} are equivalent to
\begin{align}
\label{eqn:lim elli}
&\lim_{n \to \infty} \frac{1}{\sqrt{n}} \ell_{i}^{(n)} = \widetilde{\ell}_{i}, &&i =d,u, \\
&\lim_{n \to \infty} \sqrt{n} (1 - \rho_{j}^{(n)}) = -b_{j}, &&j =1,2. \nonumber
\end{align}
Under these conditions, we can now state the main result of this paper. 

Let $\nu^{(n)}$ denote the probability measure of the scaled queue length $\dfrac{1}{\sqrt{n}} L^{(n)}$, that is,
\begin{align*}
\nu^{(n)}(E) = \mathbb{P}\left(\frac{1}{\sqrt{n}} L^{(n)} \in E \right), \qquad n \in \mathbb{N},\ E \in {\sr B}(\mathbb{R}),
\end{align*}
where ${\sr B}(\mathbb{R})$ is the Borel $\sigma$-field on $\mathbb{R}$.
We are now in a position to state our main result.

\begin{theorem}
\label{thr:main}
Suppose that Assumptions \enu{stability}--\enu{rho} hold. Then, as $n \to \infty$, the sequence $\{\nu^{(n)}\}$ converges weakly to a probability measure $\nu$ with a probability density function $f = f_{1,1} + f_{2,1} + f_{1,2} + f_{2,2}$, that $f_{1,1}, f_{2,1}, f_{1,2}, f_{2,2}$ are defined as follows:
\begin{align}
\label{eqn:con den 1,l}
&f_{1,1}(x) =  C_{0} \left(b_{1} (e^{b_{1}(\widetilde{\ell}_{u} - \widetilde{\ell}_{d})}-1 )1(b_{1} \neq 0) + 1(b_{1} = 0)  \right) e^{b_{1}x}1(x \in [0,\widetilde{\ell}_{d})),\\
\label{eqn:con den 1,u}
&f_{2,1}(x) =  C_{0} \left( b_{1}(e^{b_{1} \widetilde{\ell}_{u}} -e^{b_{1} x})1(b_{1} \neq 0)  +  \frac{\widetilde{\ell}_{u} - x}{\widetilde{\ell}_{u} - \widetilde{\ell}_{d}} 1(b_{1} = 0)\right)1(x \in [\widetilde{\ell}_{d},\widetilde{\ell}_{u}]),\\
\label{eqn:con den 2,l}
&f_{1,2}(x) =  \frac{C_{0}\widetilde{\rho}_{12}}{b_{2}} \left(b_{1}^{2} e^{b_{1} \widetilde{\ell}_{u}}1(b_{1} \neq 0) + \frac{1(b_{1} = 0)}{\widetilde{\ell}_{u} - \widetilde{\ell}_{d}} \right)(e^{b_{2} (x - \widetilde{\ell}_{d})} - 1) 1(x \in [\widetilde{\ell}_{d},\widetilde{\ell}_{u}]),\\
\label{eqn:con den 2,u}
&f_{2,2}(x) = \frac{C_{0}\widetilde{\rho}_{12}(e^{b_{2}(\widetilde{\ell}_{u} - \widetilde{\ell}_{d})} - 1)}{b_{2}} \left(b_{1}^{2}e^{b_{1} \widetilde{\ell}_{u}}1(b_{1} \neq 0) +  \frac{1(b_{1} = 0)}{\widetilde{\ell}_{u} - \widetilde{\ell}_{d}} \right)e^{b_{2}(x -\widetilde{\ell}_{u})}  1(x \in (\widetilde{\ell}_{u},\infty)),
\end{align}
where $C_{0}$ is the normalization constant given by
\begin{align}
\label{eqn:constant C0}
C_{0} = 
\begin{cases}
\left(1 - e^{b_{1}(\widetilde{\ell}_{u} - \widetilde{\ell}_{d})} + b_{1}(\widetilde{\ell}_{u} - \widetilde{\ell}_{d})e^{b_{1}\widetilde{\ell}_{u}}- \dfrac{b_{1}^{2}}{b_{2}}(\widetilde{\ell}_{u} - \widetilde{\ell}_{d})\widetilde{\rho}_{12} e^{b_{1} \widetilde{\ell}_{u}}\right)^{-1}, & b_{1} \neq 0, \\
\left(\dfrac{\widetilde{\ell}_{u} + \widetilde{\ell}_{d}}{2} - \dfrac{\widetilde{\rho}_{12}}{b_{2}} \right)^{-1}, & b_{1} = 0.
\end{cases}
\end{align}
\end{theorem}

\begin{remark}
\label{rem:exp uni}
For $i,j = 1,2$, the density function $f_{i,j}$ reveals several interesting properties. For $b_1 \neq 0$, $f_{1,1}$ and $f_{2,2}$ are exponential in form, while for $b_1=0$, $f_{1,1}$ corresponds to a uniform density. 
In this respect, the distribution $\nu$ is similar to that found in \cite{KobaMiyaSaku2025}.
However, the components $f_{2,1}$ and $f_{1,2}$ are neither exponential nor uniform. 
These novel forms are a direct consequence of the history-dependent structure of our model.
For the level-dependent queue in \cite{KobaMiyaSaku2025}, which lacks the overlapping region $\{\ell_d,\ell_d+1, \dots, \ell_u\}$ between two levels, the corresponding density components vanish.
\end{remark}

We prove \thr{main} in the next section.

\section{Proof of \thr{main}}
\setnewcounter
\label{sec:proof}
For each $n \in \mathbb{N}$ and $E \in {\sr B}(\mathbb{R})$, let $\nu_{1,1}^{(n)},\nu_{2,1}^{(n)},\nu_{1,2}^{(n)}$ and  $\nu_{2,2}^{(n)}$ denote the measures corresponding to the four state space partitions, defined by
\begin{align*}
\nu_{i,j}^{(n)}(E) = \mathbb{P}\left(\frac{1}{\sqrt{n}}L^{(n)} \in E,(L^{(n)},B^{(n)}) \in S_{i,j}^{(n)} \right), \qquad i,j = 1,2,
\end{align*}
 where we recall that 
\begin{align*}
&S_{1,1}^{(n)} = \{0,1,\dots,\ell_{d}^{(n)}-1\} \times \{1\}, && S_{2,1}^{(n)} = {\{\ell_{d}^{(n)},\ell_{d}^{(n)}+1,\dots,\ell_{u}^{(n)}\}} \times \{1\},\\
&S_{1,2}^{(n)} = {\{\ell_{d}^{(n)},\ell_{d}^{(n)}+1,\dots,\ell_{u}^{(n)}\}} \times \{2\}, 
&&S_{2,2}^{(n)} = \{\ell_{u}^{(n)}+1,\ell_{u}^{(n)} + 2, \dots \} \times \{2\}.
\end{align*}
For any $n \in \mathbb{N}$ and $E \in {\sr B}(\mathbb{R})$, it is clear that these measures partition the total probability measure $\nu^{(n)}$:
\begin{align}
\label{eqn:partition}
\nu^{(n)}(E) = \sum_{i,j =1,2}\nu_{i,j}^{(n)}(E).
\end{align}
For $i,j = 1,2$, let $\nu_{i,j}$ denote the measure with density function $f_{i,j}$, that is,
\begin{align*}
\nu_{i,j} (E) = \int_{E} f_{i,j}(x) dx, \qquad E \in {\sr B}(\mathbb{R}).
\end{align*}
To prove Theorem \ref{thr:main}, it is sufficient to show that
\begin{align}
\label{eqn:lim nu i,j}
\lim_{n \to \infty} \nu_{i,j}^{(n)}(E) = \nu_{i,j}(E), \qquad i,j = 1,2, E \in {\sr B}(\mathbb{R}).
\end{align}
To verify \eqn{lim nu i,j}, it suffices to prove $\varphi_{i,j}^{(n)} \to \varphi_{i,j}$ as $n \to \infty$ where $\varphi_{i,j}^{(n)}$ and $\varphi_{i,j}$ are defined as
\begin{align*}
&\varphi_{i,j}^{(n)}(\theta) = \int_{-\infty}^{\infty} e^{\theta x} \nu_{i,j}^{(n)}(dx) = \mathbb{E}\left(e^{\frac{\theta}{\sqrt{n}}L^{(n)}} 1 ({(L^{(n)},B^{(n)}) \in S_{i,j}^{(n)}}) \right),\\
&\varphi_{i,j}(\theta) = \int_{-\infty}^{\infty} e^{\theta x} \nu_{i,j} (dx) = \int_{-\infty}^{\infty} e^{\theta x} f_{i,j}(x) dx,
\end{align*}
for $n \in \mathbb{N}, i,j = 1,2$ and $\theta \in \mathbb{R}$.

It follows from \eqn{psi 11}, \eqn{psi 11 1} and  Assumptions \enu{ell con} and \enu{rho} that, for $\theta \neq -b_{1}$,  
\begin{align}
\label{eqn:varphi n 11}
\varphi_{1,1}^{(n)}(\theta) &= \psi_{1,1}^{(n)}\left(\frac{\theta}{\sqrt{n}} \right) \nonumber\\
&= \frac{1 - (\rho_{1}^{(n)}e^{\frac{\theta}{\sqrt{n}}})^{\ell_{d}^{(n)} + 1}}{1 - \rho_{1}^{(n)} e^{\frac{\theta}{\sqrt{n}}}}\pi^{(n)}(0,1) \nonumber\\
&= \frac{e^{(\theta + b_{1})\widetilde{\ell_{d}}}-1  + o(1)}{\theta + b_{1} + o(1)} \sqrt{n} \pi^{(n)}(0,1),
\end{align}
where we use the asymptotic expansions
\begin{align}
\label{eqn:rho1 b1}
\left(\rho_{1}^{(n)}\right)^{\ell_{d}^{(n)} + 1} = e^{b_{1} \widetilde{\ell}_{d}} + o(1), \qquad \rho_{1}^{(n)} e^{\frac{\theta}{\sqrt{n}}} = \left(1 + \frac{b_{1}}{\sqrt{n}} + o\left(\frac{1}{\sqrt{n}} \right) \right)\left(1 + \frac{\theta}{\sqrt{n}} + o\left(\frac{1}{\sqrt{n}} \right) \right).
\end{align}
The expression \eqn{varphi n 11} shows that to prove $\varphi_{i,j}^{(n)}(\theta) \to \varphi_{i,j}(\theta)$ as $n \to \infty$, we need to determine the asymptotic behavior of $\sqrt{n} \pi^{(n)}(0,1)$. 
For this, let $g_{i}(n)$ be a function on $\mathbb{N}$ representing the higher-order error term in the heavy traffic expansion:
\begin{align}
\label{eqn:g1n}
g_{i}(n) = \rho_{i}^{(n)} - 1 - \dfrac{b_{i}}{\sqrt{n}}, \qquad i = 1,2.
\end{align}
From Assumption \enu{rho}, we immediately have $g_{i}(n) = o\left(\dfrac{1}{\sqrt{n}}\right)$ as $n \to \infty$. 
We first obtain the following lemma.

\begin{lemma}
\label{lem:expand lem}
Let $a(n) \in \mathbb{R}$ and $b(n) > 0$ be functions on $\mathbb{N}$ satisfying $a(n) = o\left(\dfrac{1}{b(n)} \right)$ and $b(n) \to \infty$ as $n \to \infty$. Then, we have 
\begin{align}
\label{eqn:expand 1}
&(1 + a(n))^{b(n)} = 1 + a(n) b(n) (1 + o(1)),\\
\label{eqn:expand 2}
&(1 + a(n))^{b(n)} = 1 + a(n) b(n)\left(1  + \frac{a(n)(b(n)-1)}{2}(1 + o(1))\right),
\end{align}
as $n \to \infty$.
\end{lemma}

\begin{proof}
Since $|a(n)| = o\left(\dfrac{1}{b(n)}\right)$ and $b(n) |\log(1 + a(n))| \le |a(n)|b(n)$, we have $\displaystyle{\lim_{n \to \infty} b(n) \log (1 + a(n)) = 0}$, which is equivalent to $\displaystyle{\lim_{n \to \infty} (1 + a(n))^{b(n)} = 1}$, i.e.,
\begin{align*}
(1 + a(n))^{b(n)} = 1 + o(1),
\end{align*}
as $n \to \infty$. From the Binomial theorem, 
\begin{align*}
(1 + a(n))^{b(n)} - 1 = \sum_{k=1}^{b(n)} \frac{b(n)!}{k! (b(n)-k)!} (a(n))^{k}.
\end{align*}
Hence, we have 
\begin{align}
\label{eqn:sum o(1)}
\sum_{k=1}^{b(n)} \frac{b(n)!}{k! (b(n)-k)!} (a(n))^{k} = o(1).
\end{align}
Factoring out the first term gives
\begin{align*}
\sum_{k=1}^{b(n)} \frac{b(n)!}{k! (b(n)-k)!} (a(n))^{k} &= a(n)b(n)\left(1 + \sum_{k=2}^{b(n)} \frac{(b(n)-1)!}{k! (b(n)-k)!} (a(n))^{k-1} \right)\\
&= a(n)b(n)\left(1 + \sum_{k=1}^{b(n)-1} \frac{(b(n)-1)!}{(k+1)! (b(n)-1-k)!} (a(n))^{k} \right).
\end{align*}
From \eqn{sum o(1)}, the summation on the right-hand side is $o(1)$, which proves \eqn{expand 1}.
Similarly, to obtain the second-order term, we have
\begin{align*}
&\sum_{k=1}^{b(n)} \frac{b(n)!}{k! (b(n)-k)!} (a(n))^{k} \\
& \quad \qquad = a(n)b(n)\left(1  +\frac{a(n) (b(n)-1)}{2}\left(1 + 2 \sum_{k=1}^{b(n)-2} \frac{(b(n)-2)!}{(k+2)! (b(n)-2-k)!} (a(n))^{k} \right) \right).
\end{align*}
From \eqn{sum o(1)} again, the summation term on the right-hand side is $o(1)$, which implies \eqn{expand 2}.
\end{proof}

From \lem{expand lem}, we have the following result for the asymptotic behavior of the normalization constant.
\begin{lemma}
\label{lem:n pin 0}
Suppose Assumptions \enu{stability}--\enu{rho} hold. If $b_{1} \neq 0$, then 
\begin{align}
\label{eqn:b1 neq 0}
\sqrt{n} \pi^{(n)}(0,1) = \left(\dfrac{1}{-b_{1} + o(1)} - \dfrac{e^{b_{1}\widetilde{\ell}_{u}}(\widetilde{\ell}_{u} - \widetilde{\ell}_{d}) + o(1)}{1 - e^{b_{1}(\widetilde{\ell}_{u} - \widetilde{\ell}_{d})} + o(1) } + \dfrac{b_{1}(\widetilde{\ell}_{u}  - \widetilde{\ell}_{d})\widetilde{\rho}_{12}e^{b_{1}\widetilde{\ell}_{u}} + o(1) }{b_{2}(1 - e^{b_{1}(\widetilde{\ell}_{u} - \widetilde{\ell}_{d})})+o(1)} \right)^{-1},
\end{align}
and if $b_{1} = 0$, then 
\begin{align}
\label{eqn:b1 = 0}
\sqrt{n} \pi^{(n)}(0,1) = \left(\frac{\widetilde{\ell}_{u} + \widetilde{\ell}_{d} + o(1)}{2(1 + o(1))} - \dfrac{\widetilde{\rho}_{12} + o(1)}{b_{2} + o(1)} \right)^{-1},
\end{align}
as $n \to \infty$. 
\end{lemma}

\begin{proof}
From Assumptions \enu{basic} and \enu{rho}, we have the following convergences:
\begin{align}
\label{eqn:basic eq1}
&\rho_{12}^{(n)} = \widetilde{\rho}_{12} + o(1), &&1 -\rho_{i}^{(n)} = - \frac{b_{i}}{\sqrt{n}} - g_{i}(n), \quad i = 1,2, \\
\label{eqn:d-1 eq}
&(\rho_{1}^{(n)})^{\ell_{u}^{(n)} - \ell_{d}^{(n)} + 2} = e^{b_{1}(\widetilde{\ell}_{u} - \widetilde{\ell}_{d})} + o(1), &&(\rho_{1}^{(n)})^{\ell_{u}^{(n)} } = (\rho_{1}^{(n)})^{\ell_{u}^{(n)} + 1} = e^{b_{1}\widetilde{\ell}_{u}} + o(1),
\end{align}
 as $n \to \infty$. 
 Moreover, from Remark \ref{rem:closed form}, for $\rho_{1}^{(n)} \neq 1$, $\pi^{(n)}(0,1)$ has the following form:
\begin{align}
\label{eqn:rho neq 1}
 \left(\dfrac{1}{1-\rho_{1}^{(n)}} - \dfrac{(\rho_{1}^{(n)})^{\ell_{u}^{(n)} + 1}(\ell_{u}^{(n)} - \ell_{d}^{(n)} +2)}{(1 - (\rho_{1}^{(n)})^{\ell_{u}^{(n)} - \ell_{d}^{(n)} + 2})} + \dfrac{\rho_{12}^{(n)}((\rho_{1}^{(n)})^{\ell_{u}^{(n)}} - (\rho_{1}^{(n)})^{\ell_{u}^{(n)} + 1})(\ell_{u}^{(n)}  - \ell_{d}^{(n)} + 2) }{(1-(\rho_{1}^{(n)})^{\ell_{u}^{(n)} - \ell_{d}^{(n)}  + 2})(1 - \rho_{2}^{(n)})} \right)^{-1}.
\end{align}
If $b_{1} \neq 0$, then $\rho_{1}^{(n)} \neq 1$ for all sufficiently large $n$. Substituting the asymptotic relations in Assumption \enu{ell con} and the fact that $\sqrt{n} g_{1}(n) = o(1)$ into \eqn{rho neq 1} gives
\begin{align*}
\pi^{(n)}(0,1) = \frac{1}{\sqrt{n}}\left(\dfrac{1}{-b_{1} + o(1)} - \dfrac{e^{b_{1}\widetilde{\ell}_{u}}(\widetilde{\ell}_{u} - \widetilde{\ell}_{d}) + o(1)}{1 - e^{b_{1}(\widetilde{\ell}_{u} - \widetilde{\ell}_{d})} + o(1) } + \dfrac{b_{1}(\widetilde{\ell}_{u}  - \widetilde{\ell}_{d})\widetilde{\rho}_{12}e^{b_{1}\widetilde{\ell}_{u}} + o(1)}{b_{2}(1 - e^{b_{1}(\widetilde{\ell}_{u} - \widetilde{\ell}_{d})})+o(1)} \right)^{-1},
\end{align*}
which completes the proof of \eqn{b1 neq 0}. 

If $b_{1} = 0$, then we must consider both cases $\rho_{1}^{(n)} \neq 1$ and $\rho_{1}^{(n)} = 1$, and we prove that they yield the same limit \eqn{b1 = 0}.
We first assume $b_{1} = 0$ and $\rho_{1}^{(n)} \neq 1$ for sufficiently large $n \in \mathbb{N}$.
Note that $\rho_{1}^{(n)} = 1 + g_{1}(n)$ with $g_{1}(n) = o\left(\dfrac{1}{\sqrt{n}} \right)$. Therefore, from \eqn{expand 1} and \eqn{expand 2} in \lem{expand lem}, 
\begin{align}
\label{eqn:d-2 eq2}
&(\rho_{1}^{(n)})^{\ell_{u}^{(n)}} = 1 + \ell_{u}^{(n)}g_{1}(n) \left(1 + o(1) \right),\\
\label{eqn:d-2 eq1}
&1 - (\rho_{1}^{(n)})^{\ell_{u}^{(n)} - \ell_{d}^{(n)} + 2} 
 = - (\ell_{u}^{(n)} - \ell_{d}^{(n)} + 2)g_{1}(n) \left(1  + \frac{\ell_{u}^{(n)} -\ell_{d}^{(n)} + 1}{2}g_{1}(n)(1 + o(1))\right).
\end{align}
Note that the $o(1)$ terms in \eqn{d-2 eq2} and \eqn{d-2 eq1} are not necessarily identical.
From \eqn{rho neq 1}, we obtain, as $n \to \infty$, 
\begin{align*}
\pi^{(n)}(0,1) &= \left(\frac{1}{-g_{1}(n)} + \frac{1 + \ell_{u}^{(n)} g_{1}(n)(1  + o(1))}{g_{1}(n)\left(1  + \frac{\ell_{u}^{(n)} -\ell_{d}^{(n)} + 1}{2}g_{1}(n)(1  + o(1))\right)} \right.\\
& \qquad \qquad \left. - \frac{\sqrt{n}g_{1}(n)(\widetilde{\rho}_{12} + o(1))(1 + \ell_{u}^{(n)}g_{1}(n) \left(1 + o(1) \right))}{g_{1}(n) \left(1  + \frac{\ell_{u}^{(n)} -\ell_{d}^{(n)} + 1}{2}g_{1}(n)(1 + o(1))\right) (b_{2} + o(1))}\right)^{-1}\\
&= \left(\frac{\ell_{u}^{(n)} + \ell_{d}^{(n)} + \sqrt{n} o(1)}{2\left(1  + \frac{\ell_{u}^{(n)} -\ell_{d}^{(n)} + 1}{2}g_{1}(n)(1  + o(1))\right)} 
 - \frac{\sqrt{n}(\widetilde{\rho}_{12} + o(1))(1 + o(1)\left(1 + o(1) \right))}{ \left(1  + o(1)(1 + o(1))\right) (b_{2} + o(1))}\right)^{-1}\\
&= \left(\frac{\sqrt{n}\left(\left(\widetilde{\ell}_{u} + \widetilde{\ell}_{d} + g_{1}(n)\right) + o(1)\right)}{2(1 + o(1))} - \frac{\sqrt{n} (\widetilde{\rho}_{12} + o(1))}{b_{2} + o(1)} \right)^{-1},
\end{align*}
which implies \eqn{b1 = 0}. 

Finally, if $\rho_{1}^{(n)} = 1$ for all sufficiently large $n$,  then from Remark \ref{rem:closed form} again, 
\begin{align*}
\pi^{(n)} (0,1) &= \left(\dfrac{\ell_{u}^{(n)} + \ell_{d}^{(n)} + 1}{2} + \dfrac{\rho_{12}^{(n)}}{1 - \rho_{2}^{(n)}} \right)^{-1}\\
&=\left(\dfrac{\sqrt{n}(\widetilde{\ell}_{u} + \widetilde{\ell}_{d} + o(1))}{2} - \dfrac{\sqrt{n}(\widetilde{\rho}_{12} + o(1))}{b_{2} + o(1)} \right)^{-1}.
\end{align*}
Hence, we have \eqn{b1 = 0} in this case as well.
\end{proof}

\begin{proof*}{Proof of \thr{main}}
From \eqn{b1 neq 0} in \lem{n pin 0}, for $b_{1} \neq 0$, 
\begin{align}
\label{eqn:pi n 0 b1 neq 0}
\lim_{n \to \infty} \sqrt{n} \pi^{(n)}(0,1) &= \left(\dfrac{1}{-b_{1}} - \dfrac{e^{b_{1}\widetilde{\ell}_{u}}(\widetilde{\ell}_{u} - \widetilde{\ell}_{d})}{1 - e^{b_{1}(\widetilde{\ell}_{u} - \widetilde{\ell}_{d})}} + \dfrac{b_{1}(\widetilde{\ell}_{u}  - \widetilde{\ell}_{d})\widetilde{\rho}_{12}e^{b_{1}\widetilde{\ell}_{u}} }{b_{2}(1 - e^{b_{1}(\widetilde{\ell}_{u} - \widetilde{\ell}_{d})})} \right)^{-1} \nonumber\\
&=  C_{0}b_{1}(e^{b_{1}(\widetilde{\ell}_{u} - \widetilde{\ell}_{d})}-1),
\end{align}
and from \eqn{b1 = 0}, for $b_{1} = 0$, 
\begin{align}
\label{eqn:pi n 0 b1 = 0}
\lim_{n \to \infty} \sqrt{n} \pi^{(n)}(0,1) = \left(\dfrac{\widetilde{\ell}_{u} + \widetilde{\ell}_{d}}{2} - \dfrac{\widetilde{\rho}_{12}}{b_{2}} \right)^{-1} = C_{0}. 
\end{align}
Thus, from \eqn{varphi n 11}, 
\begin{align*}
\lim_{n \to \infty} \varphi_{1,1}^{(n)}(\theta) =
C_{0} (b_{1}(e^{b_{1}(\widetilde{\ell}_{u} - \widetilde{\ell}_{d})}-1)1(b_{1} \neq 0) + 1(b_{1} = 0)) \frac{e^{(\theta + b_{1})\widetilde{\ell_{d}}}-1}{\theta + b_{1}}.
\end{align*}
On the other hand, it follows from \eqn{con den 1,l} that 
\begin{align*}
\varphi_{1,1}(\theta) &= \int_{0}^{\widetilde{\ell}_{d}}  e^{\theta x}C_{0} \left(b_{1} (e^{b_{1}(\widetilde{\ell}_{u} - \widetilde{\ell}_{d})}-1 )1(b_{1} \neq 0) + 1(b_{1} = 0)  \right) e^{b_{1}x}  dx\\
&=C_{0} (b_{1}(e^{b_{1}(\widetilde{\ell}_{u} - \widetilde{\ell}_{d})}-1)1(b_{1} \neq 0) + 1(b_{1} = 0)) \frac{e^{(\theta + b_{1})\widetilde{\ell_{d}}}-1}{\theta + b_{1}}.
\end{align*}
Thus, we have shown that $\displaystyle{\lim_{n \to \infty} \varphi_{1,1}^{(n)}(\theta) = \varphi_{1,1}(\theta)}$, which implies that $\displaystyle{\lim_{n \to \infty} \nu_{1,1}^{(n)}(E) = \nu_{1,1}(E)}$ for $E \in \sr{B}(\mathbb{R})$.
The convergence for the other components can be verified as follows:
\begin{align}
\label{eqn:comp 21}
&\lim_{n \to \infty} \varphi_{2,1}^{(n)}(\theta)   = C_{0}b_{1} \left(\frac{e^{(\theta + b_{1})\widetilde{\ell}_{d}} - e^{(\theta + b_{1})\widetilde{\ell}_{u}}}{\theta + b_{1}} + e^{b_{1} \widetilde{\ell}_{u}}\frac{e^{\theta \widetilde{\ell}_{u}} - e^{\theta \widetilde{\ell}_{d}}}{\theta} \right)1(b_{1} \neq 0) \nonumber\\ 
& \qquad \qquad \qquad \qquad + C_{0}\frac{\theta e^{\theta \tilde{\ell}_{d}}\left(\widetilde{\ell}_{d} - \widetilde{\ell}_{u}\right) + e^{\theta \widetilde{\ell}_{u}} - e^{\theta \widetilde{\ell}_{d}}}{\theta^{2}(\widetilde{\ell}_{u} - \widetilde{\ell}_{d})}1(b_{1} = 0) = \varphi_{2,1}(\theta),\\
\label{eqn:comp 12}
&\lim_{n \to \infty} \varphi_{1,2}^{(n)}(\theta) =  \frac{C_{0}\widetilde{\rho}_{12}}{b_{2}} b_{1}^{2} e^{b_{1} \widetilde{\ell}_{u}} \left(\frac{e^{\theta \widetilde{\ell}_{d}} - e^{\theta \widetilde{\ell}_{u} }}{\theta} +  \frac{ e^{\theta\widetilde{\ell}_{u}  + b_{2}(\widetilde{\ell}_{u}- \widetilde{\ell}_{d})} - e^{\theta \widetilde{\ell}_{d}}}{\theta + b_{2}} \right)1(b_{1} \neq 0) \nonumber\\
& \qquad \qquad \qquad \qquad +\dfrac{C_{0}\widetilde{\rho}_{12}}{b_{2}(\widetilde{\ell}_{u} -\widetilde{\ell}_{d})} \left(\frac{e^{\theta \widetilde{\ell}_{d}} - e^{\theta \widetilde{\ell}_{u} }}{\theta} + \frac{ e^{\theta\widetilde{\ell}_{u}  + b_{2}(\widetilde{\ell}_{u}- \widetilde{\ell}_{d})} - e^{\theta \widetilde{\ell}_{d}}}{\theta + b_{2}} \right)1(b_{1} = 0) = \varphi_{1,2}(\theta),\\
\label{eqn:comp 22}
&\lim_{n \to \infty} \varphi_{2,2}^{(n)}(\theta) = \frac{C_{0}\widetilde{\rho}_{12}}{b_{2}} b_{1}^{2} e^{b_{1} \widetilde{\ell}_{u}} (1 - e^{b_{2}(\widetilde{\ell}_{u}- \widetilde{\ell}_{d})})  \frac{e^{\theta \widetilde{\ell}_{u}}}{\theta + b_{2}} 1(b_{1} \neq 0) \nonumber\\
& \qquad \qquad \qquad \qquad + \dfrac{C_{0}\widetilde{\rho}_{12}}{b_{2}(\widetilde{\ell}_{u} -\widetilde{\ell}_{d})} (1 - e^{b_{2}(\widetilde{\ell}_{u}- \widetilde{\ell}_{d})})  \frac{e^{\theta \widetilde{\ell}_{u}}}{\theta + b_{2}} 1(b_{1} =0)  = \varphi_{2,2}(\theta).
\end{align}
Since \eqn{comp 12} and \eqn{comp 22} are derived similarly, we only prove \eqn{comp 21} in \app{det comp}.
Thus, for all $i,j = 1,2$, we have $\displaystyle{\lim_{n \to \infty} \varphi_{i,j}^{(n)}(\theta) = \varphi_{i,j}(\theta)}$, which completes the proof of \thr{main}.
\end{proof*}

\section{Approximations and numerical examples}
\setnewcounter
\label{sec:app and num}

In this section, we derive approximation formulas from our main theorem and evaluate their accuracy using numerical examples. 
We first obtain the convergence of the mean.
\begin{lemma}
\label{lem:mean dif}
Under the same  assumptions of \thr{main}, 
\begin{align}
\label{eqn:mean dif}
\lim_{n \to \infty} \mathbb{E}\left(\frac{1}{\sqrt{n}} L^{(n)}\right) &= \int_{0}^{\infty} x \nu(dx),
\end{align}
where the integral on the right-hand side is the Lebesgue integral with respect to the probability measure $\nu$. 
\end{lemma}
\begin{proof}
From \thr{main}, it is enough to verify the uniform integrability of the sequence of random variables $\left\{\dfrac{1}{\sqrt{n}} L^{(n)}\right\}$.
From Assumption \enu{ell con}, for any $\alpha \ge \widetilde{\ell_{u}} + 1$ and sufficiently large $n$, we have $(L^{(n)}, B^{(n)}) \in S_{2,2}^{(n)}$ if $\dfrac{1}{\sqrt{n}} L^{(n)} \ge \alpha$. Thus,
\begin{align*}
\lim_{n\to\infty} \mathbb{P}\left(\dfrac{1}{\sqrt{n}} L^{(n)} \ge \alpha \right) = \lim_{n\to\infty} \mathbb{P}\left(\dfrac{1}{\sqrt{n}} L^{(n)} \ge \alpha,   (L^{(n)}, B^{(n)}) \in S_{2,2}^{(n)}\right).
\end{align*}
Furthermore, by \rem{closed form}, the stationary distribution for the $n$-th system has the form
\begin{align*}
\pi^{(n)}(\ell,2) = c(n)(\rho_{2}^{(n)})^{\ell - \ell_{u}^{(n)}+1}, \qquad \ell \ge \ell_{u}^{(n)},
\end{align*} 
where $c(n)$ is given by
\begin{align*}
\left(\dfrac{(\rho_{1}^{(n)})^{\ell_{u}^{(n)}} - (\rho_{1}^{(n)})^{\ell_{u}^{(n)}+1}}{1 - (\rho_{1}^{(n)})^{\ell_{u}^{(n)} - \ell_{d}^{(n)} + 2}}1(\rho_{1}^{(n)} \neq 1) +  \dfrac{1(\rho_{1}^{(n)} = 1)}{\ell_{u}^{(n)}- \ell_{d}^{(n)} + 2}\right)\dfrac{1 - (\rho_{2}^{(n)})^{\ell_{u}^{(n)} - \ell_{d}^{(n)}}}{1 - \rho_{2}^{(n)}} \rho_{12}^{(n)} \pi^{(n)}(0,1).
\end{align*}
Hence, 
\begin{align*}
\mathbb{E}\left(\dfrac{1}{\sqrt{n}} L^{(n)}1\left(\dfrac{1}{\sqrt{n}} L^{(n)}\ge \alpha \right) \right) &= \frac{1}{\sqrt{n}} \sum_{\ell=\lceil \sqrt{n} \alpha \rceil}^{\infty} \ell \pi^{(n)}(\ell,2) \\
& = \sqrt{n}c(n) \left(\frac{\lceil \sqrt{n} \alpha \rceil}{n(1 - \rho_{2}^{(n)})}(\rho_{2}^{(n)})^{\lceil \sqrt{n} \alpha \rceil}  + \frac{(\rho_{2}^{(n)})^{\lceil \sqrt{n} \alpha \rceil + 1}}{n(1 - \rho_{2}^{(n)})^{2}}\right),
\end{align*}
where $\lceil x \rceil = \min\{y \in \mathbb{Z}; y \ge x\}$ for $x \in \mathbb{R}$.
It follows from \lem{n pin 0} and the proof of \thr{main} that $\sqrt{n} c(n)$ converges to a positive constant. Thus, there exists a constant $K$ such that 
\begin{align*}
\sup_{n \in \mathbb{N}} \sqrt{n} c(n) < K.
\end{align*} 
For sufficiently large $n_{0}$, we have for $n \ge n_{0}$,
\begin{align}
\label{eqn:uni two eq1}
&\frac{\lceil \sqrt{n} \alpha \rceil}{n(1 - \rho_{2}^{(n)})}(\rho_{2}^{(n)})^{\lceil \sqrt{n} \alpha \rceil} = \frac{(1 + o(1)) \lceil \alpha \rceil }{-b_{2} + o(1)}(e^{b_{2}\alpha} + o(1)),\\
\label{eqn:uni two eq2}
&\frac{(\rho_{2}^{(n)})^{\lceil \sqrt{n} \alpha \rceil + 1}}{n(1 - \rho_{2}^{(n)})^{2}} = \frac{e^{b_{2}\alpha} + o(1)}{\left(-b_{2} + o(1) \right)^{2}}.
\end{align}
Since $b_{2} < 0$ under Assumption \enu{rho}, for any $\epsilon > 0$, there exists a $\beta$ such that the right-hand sides of \eqn{uni two eq1} and \eqn{uni two eq2} are not greater than $\dfrac{\epsilon}{2K}$ for any $\alpha \ge \beta$.
In addition, by the condition $\rho_{2}^{(n)} < 1$, we can choose a sufficiently large $\gamma$ such that, for $\alpha \ge \gamma$ and all $n < n_{0}$,  
\begin{align*}
\frac{\lceil \sqrt{n} \alpha \rceil}{n(1 - \rho_{2}^{(n)})}(\rho_{2}^{(n)})^{\lceil \sqrt{n} \alpha \rceil} < \frac{\epsilon}{2K}, \quad \quad \frac{(\rho_{2}^{(n)})^{\lceil \sqrt{n} \alpha \rceil + 1}}{n(1 - \rho_{2}^{(n)})^{2}} < \frac{\epsilon}{2K}.
\end{align*}
Combining these observations, we can conclude that
\begin{align*}
\lim_{\alpha \to \infty} \sup_{n \in \mathbb{N}} \mathbb{E}\left(\dfrac{1}{\sqrt{n}} L^{(n)}1\left(\dfrac{1}{\sqrt{n}} L^{(n)}\ge \alpha \right) \right) = 0.
\end{align*}
We complete the proof.
\end{proof}

Moreover, from \eqn{pi n 0 b1 neq 0}, \eqn{pi n 0 b1 = 0}, and \eqn{mean dif}, we readily obtain the following result.
\begin{corollary}
\label{cor:asy Ln}
The mean queue length $\mathbb{E}(L^{(n)})$ has the following asymptotic property:
\begin{align}
\label{eqn:mean app}
\lim_{n \to \infty} \pi^{(n)}(0,1) \mathbb{E}(L^{(n)}) &= \lim_{n \to \infty} \sqrt{n} \pi^{(n)}(0,1) \mathbb{E}\left(\frac{1}{\sqrt{n}}L^{(n)} \right) \nonumber\\
&= C_{0}(b_{1} (e^{b_{1}(\widetilde{\ell}_{u} - \widetilde{\ell}_{d})} - 1 )1(b_{1} \neq 0) + 1(b_{1} = 0)) \int_{0}^{\infty} x \nu (dx). 
\end{align}
\end{corollary}

\cor{asy Ln} implies that the growth rate of $\mathbb{E}(L^{(n)})$ and the decay rate of the idle probability $\pi^{(n)}(0,1) = \mathbb{P}(L^{(n)} = 0) $ are the same. 

We next present numerical examples to evaluate the performance of our approximation. 
From Assumption \enu{rho}, a practical approximation for the mean queue length can be derived from \eqn{mean dif}:
\begin{align}
\label{eqn:Ln sqrtn app}
\mathbb{E}(L^{(n)}) \thickapprox \frac{-b_{2}}{1 - \rho_{2}^{(n)}} \int_{0}^{\infty} x \nu(dx),
\end{align}
where from \thr{main}, for $b_{1} \neq 0$, 
\begin{align*}
\int_{0}^{\infty} x \nu (dx) &= \int_{0}^{\widetilde{\ell}_{d}} x f_{1,1}(x) dx + \int_{\widetilde{\ell}_{d}}^{\widetilde{\ell}_{u}} x (f_{2,1}(x) + f_{1,2}(x)) dx + \int_{\widetilde{\ell}_{u}}^{\infty} x f_{2,2}(x) dx\\
&= \left(1 - e^{b_{1}(\widetilde{\ell}_{u} - \widetilde{\ell}_{d})} + b_{1}(\widetilde{\ell}_{u} - \widetilde{\ell}_{d})e^{b_{1}\widetilde{\ell}_{u}}- \dfrac{b_{1}^{2}}{b_{2}}(\widetilde{\ell}_{u} - \widetilde{\ell}_{d})\widetilde{\rho}_{12} e^{b_{1} \widetilde{\ell}_{u}}\right)^{-1}\\
& \qquad \qquad \left(\frac{1}{b_{1}}(e^{b_{1}(\widetilde{\ell}_{u} - \widetilde{\ell}_{d})} - 1) + e^{b_{1} \widetilde{\ell}_{u}}\left(\frac{b_{1}(\widetilde{\ell}_{u}^{2} - \widetilde{\ell}_{d}^{2})}{2} - \widetilde{\ell}_{u} + \widetilde{\ell}_{d} \right) \right.\\
& \qquad \qquad \left.+ \frac{b_{1}^{2}\widetilde{\rho}_{12} e^{b_{1} \widetilde{\ell}_{u}}}{b_{2}^{2}} \left(\widetilde{\ell}_{u} - \widetilde{\ell}_{d} - \frac{b_{2}(\widetilde{\ell}_{u}^{2} - \widetilde{\ell}_{d}^{2})}{2} \right)\right),
\end{align*}
and for $b_{1} = 0$, 
\begin{align*}
\int_{0}^{\infty} x \nu (dx) 
= \left(\dfrac{\widetilde{\ell}_{u} + \widetilde{\ell}_{d}}{2} - \dfrac{\widetilde{\rho}_{12}}{b_{2}} \right)^{-1}\left(\frac{\widetilde{\ell}_{u}^{2} + \widetilde{\ell}_{u}\widetilde{\ell}_{d} + \widetilde{\ell}_{d}^{2}}{6} + \frac{\widetilde{\rho}_{12}  }{b_{2}^{2}} \left(1 - \frac{b_{2}(\widetilde{\ell}_{u} + \widetilde{\ell}_{d})}{2} \right)\right).
\end{align*}
We first examine the accuracy of approximation \eqn{Ln sqrtn app} for different values of $n$.
\begin{table}[h]
\centering
\caption{Numerical accuracy of approximation \eqn{Ln sqrtn app} for $n = 10,10^{2},10^{3},10^{4}$.}
\begin{tabular}{c|cccc}
   $n$ & $10$ & $10^{2}$ & $10^{3}$ & $10^{4}$ \\
\hline 
   $\rho_{1}^{(n)}$ & $1.31623$	& $1.1$	&$1.03162$ & $1.01$\\
\hline
   $\rho_{2}^{(n)}$ & $0.683772$ & $0.9$ &$0.968377$ &$0.99$\\
\hline
   $\rho_{12}^{(n)}$ & $0.483772$ &$0.7$ &$0.768377$ &$0.79$\\
\hline
   $\mathbb{E}(L^{(n)})$ & $18.5982$ & $62.6715$ &$201.009$ & $640.299$\\
\hline
   Approximation \eqn{Ln sqrtn app} & $20.296$ &$64.1817$ &$202.96$ & $641.817$
\\
\hline
\end{tabular}
\end{table}
Table 1 summarizes the results.
We set the following parameters for this experiment:
\begin{align}
\label{eqn:set}
b_{1} = 1, \ b_{2} = -1, \ \widetilde{\ell}_{d} = 3, \ \widetilde{\ell}_{u} = 10, \ \rho_{12}^{(n)} = 0.8 - \frac{1}{\sqrt{n}} \to 0.8 \ (n \to \infty).
\end{align}
The setup set corresponds to the case $\rho_{1}^{(n)} > 1$.
As expected, the accuracy of the approximation improves as $n$ increases.
Moreover, the approximation accuracy appears to be reasonably good even for small $n$. 
\fig{graph n} (left panel) further illustrates this convergence by plotting the exact and approximate values for $n$ ranging from $1$ to $10^{3}$. The two curves are nearly indistinguishable, highlighting the high accuracy of our approximation across the entire range.

\begin{figure}[h]
 	\centering
	\includegraphics[height=0.17\textheight]{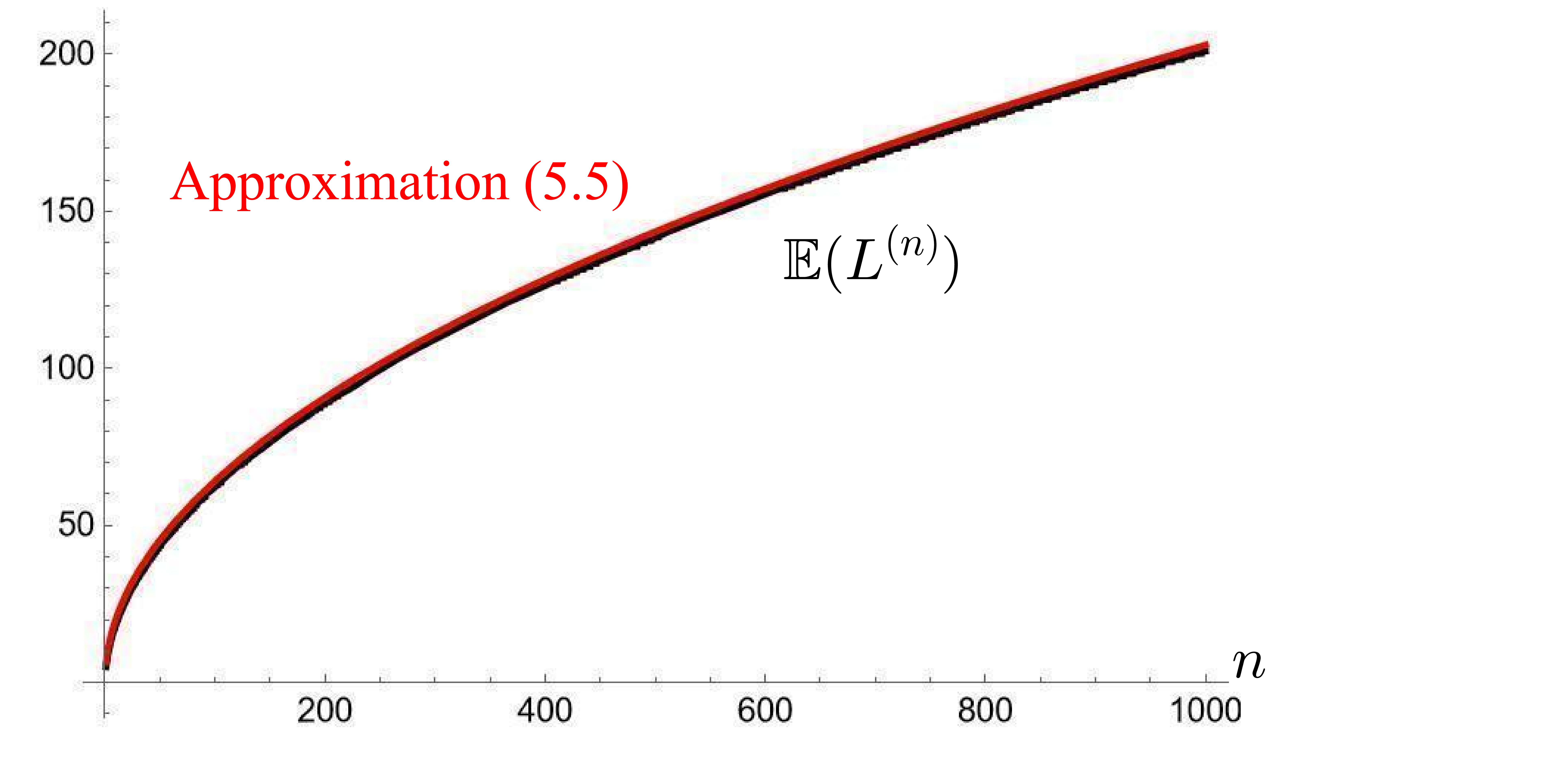}
	\includegraphics[height=0.17\textheight]{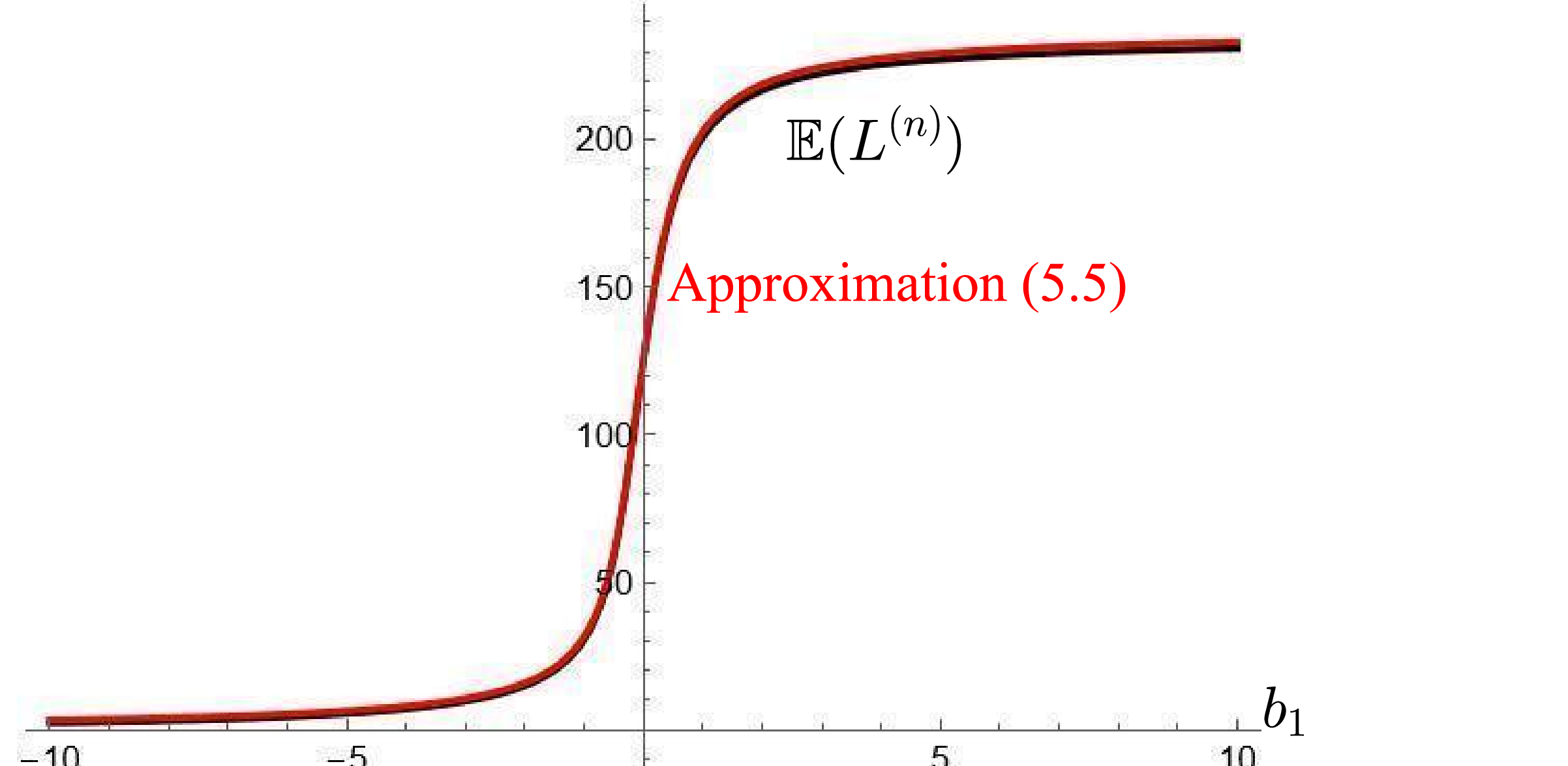}
	\caption{Graphs of changing $n$ and $b_{1}$}
	\label{fig:graph n}
\end{figure}

Finally, we investigate the effect of the parameter $b_{1}$ on the mean queue length, as shown in \fig{graph n} (right panel).
For this experiment, we fix the diffusion parameter at $n = 1000$ and vary $b_{1}$ from $-10$ to $10$. The other parameters are unchanged from \eqn{set}. From the graph, we observe that the approximation maintains excellent accuracy across the entire range of $b_1$ values.
These numerical results suggest that our approximation can be a useful and reliable tool for practical applications.

\section{Concluding remarks}
\setnewcounter
\label{sec:con}
In this paper, we analyzed the M/M/1 queue where the arrival and service rates depend on both the current queue length (level) and its recent history. We first derived the stationary distribution in closed form. We then obtained its heavy traffic diffusion limit, proving that the sequence of stationary distributions for the scaled queue length converges weakly to a limiting distribution with the density function $f$. The density function $f$ is noteworthy as some of its components are neither exponential nor uniform, a direct consequence of the history-dependent structure. Finally, we used our diffusion limit to construct a practical approximation for the mean queue length and demonstrated its effectiveness and reliability through numerical examples.

This study opens several avenues for future research. A natural next step is to extend our analysis beyond the Markovian setting. The paper \cite{KobaMiyaSaku2025} considers a multi-level GI/G/$1$ queue and obtains the diffusion limit of its stationary distribution using the stationary equation and the test function, also known as a BAR approach. We conjecture that the BAR approach can also be applied to our history-dependent framework. Therefore, investigating a history-dependent GI/G/1 queue is a promising direction for future work. Another potential extension is to consider a more complex background process with a larger state space, allowing for more intricate forms of history dependence.

\section*{Acknowledgment}
The first author (M. Kobayashi) was supported in part by JSPS KAKENHI Grant Number JP19K11845, and the third author (Y. Sakuma) was supported in part by JSPS KAKENHI Grant Number JP22K11927. This work was also supported by the Research Institute for Mathematical Sciences (RIMS), an International Joint Usage/Research Center at Kyoto University.

\appendix
\section{Proof of \thr{sta exact}}
\label{app:pi}
\setnewcounter
From the balance equation \eqn{sta1} for $k = 2$, we have
\begin{align*}
\pi(\ell,2) = \rho_{2}^{\ell - (\ell_{u}+ 1)} \pi(\ell_{u} + 1,2), \qquad \ell \ge \ell_{u} + 1.
\end{align*}
Substituting this expression into \eqn{sta6}, we have 
\begin{align*}
\pi(\ell_{u} + 1,2) =  \rho_{2} \pi(\ell_{u},2) + \rho_{12} \pi(\ell_{u},1).
\end{align*}
From \eqn{sta1} again, it follows that
\begin{align*}
\pi(\ell,2) &= \rho_{2}^{\ell - \ell_{d}} \pi(\ell_{d},2) + \sum_{j=0}^{\ell - (\ell_{d} + 1)} \rho_{2}^{j} \rho_{12} \pi(\ell_{u},1),\qquad \ell_{d} \le \ell \le \ell_{u} + 1.
\end{align*}
From \eqn{sta4} and \eqn{sta6}, we have 
\begin{align}
\label{eqn:staA1}
\pi(\ell_{d},2) = \rho_{12} \pi(\ell_{u},1).
\end{align}
On the other hand, from \eqn{sta5}, 
\begin{align*}
\pi(\ell_{u},1) = \frac{\lambda_{1}}{\lambda_{1} + \mu_{1}}\pi(\ell_{u}-1,1) = \frac{\rho_{1}}{1 + \rho_{1}}  \pi(\ell_{u}-1,1).
\end{align*}
It follows from this and \eqn{sta1} that 
\begin{align*}
\pi(\ell_{u}-1,1) = \frac{\lambda_{1}}{\lambda_{1} + \frac{\mu_{1}}{1 + \rho_{1}}} \pi(\ell_{u}-2,1) = \frac{\rho_{1} + \rho_{1}^{2}}{1 + \rho_{1} + \rho_{1}^{2}} \pi(\ell_{u}-2,1).
\end{align*}
Hence, for $\ell_{d}-1 \le \ell \le \ell_{u}$, we can inductively obtain 
\begin{align}
\label{eqn:sta7}
\pi(\ell,1) &= \frac{\rho_{1} + \rho_{1}^{2} + \dots + \rho_{1}^{\ell_{u} - (\ell - 1)}}{1 + \rho_{1} + \dots  + \rho_{1}^{\ell_{u} - (\ell-1)}} \pi(\ell-1,1) \nonumber\\
&= \frac{\rho_{1}^{2} + \rho_{1}^{3} + \dots + \rho_{1}^{\ell_{u}  -(\ell-2)}}{1 + \rho_{1} + \dots  + \rho_{1}^{\ell_{u} - (\ell-2)}} \pi(\ell-2,1)\nonumber\\
&= \dots \nonumber\\
&= \frac{\rho_{1}^{\ell - \ell_{d}+1} + \rho_{1}^{\ell - \ell_{d}+2}+ \dots + \rho_{1}^{\ell_{u} - \ell_{d}+1}}{1 + \rho_{1} + \dots +  \rho_{1}^{\ell_{u} - \ell_{d} +1}} \pi(\ell_{d}-1,1).
\end{align}
From \eqn{sta3}, \eqn{staA1} and \eqn{sta7}, we have 
\begin{align*}
&(\lambda_{1} + \mu_{1} )\pi(\ell_{d}-1,1) \\
& \qquad = \lambda_{1} \pi(\ell_{d}- 2,1) + \lambda_{1} \frac{1 + \rho_{1} + \dots + \rho_{1}^{\ell_{u} - \ell_{d}}}{1 + \rho_{1} + \dots +  \rho_{1}^{\ell_{u} - \ell_{d} + 1}} \pi(\ell_{d}-1,1) \\
& \qquad \qquad + \lambda_{1} \frac{\rho_{1}^{\ell_{u} - \ell_{d}+1}}{1 + \rho_{1} + \dots +  \rho_{1}^{\ell_{u} - \ell_{d} + 1}} \pi(\ell_{d}-1,1)\\
& \qquad = \lambda_{1} \pi(\ell_{d}- 2,1) + \lambda_{1} \pi(\ell_{d}-1,1).
\end{align*}
Hence, 
\begin{align*}
\pi(\ell_{d}-1,1)= \frac{\lambda_{1}}{\mu_{1}}\pi(\ell_{d}-2,1) = \rho_{1} \pi(\ell_{d}-2,1).
\end{align*}
Based on this result, and from \eqn{sta1} and \eqn{sta2}, we can conclude by induction that
\begin{align*}
\pi(\ell,1) = \rho_{1}^{\ell} \pi(0,1), \qquad \text{for } 0 \le \ell \le \ell_{d}-1.
\end{align*}
The full expressions in Theorem \ref{thr:sta exact} are then obtained by substituting these intermediate results back into one another. This completes the proof.

\section{Computation of \eqn{comp 21}}
\label{app:det comp}

Clearly, we first have 
\begin{align*}
\varphi_{2,1}(\theta) &= \int_{\widetilde{\ell}_{d}}^{\widetilde{\ell}_{u}} e^{\theta x}  C_{0} \left( b_{1}(e^{b_{1} \widetilde{\ell}_{u}} -e^{b_{1} x})1(b_{1} \neq 0)  +  \frac{\widetilde{\ell}_{u} - x}{\widetilde{\ell}_{u} - \widetilde{\ell}_{d}} 1(b_{1} = 0)\right) dx\\
&= \left[C_{0}b_{1}\left(e^{b_{1} \widetilde{\ell}_{u}}\frac{e^{\theta x}}{\theta} - \frac{e^{(\theta + b_{1})x}}{\theta + b_{1}}\right)1(b_{1} \neq 0) + \left(\frac{\widetilde{\ell}_{u} e^{\theta x}}{\theta(\widetilde{\ell}_{u} - \widetilde{\ell}_{d})}  - \frac{\theta x e^{\theta x} - e^{\theta x}}{\theta^{2}(\widetilde{\ell}_{u} - \widetilde{\ell}_{d})} \right)1(b_{1} = 0)\right]_{\widetilde{\ell}_{d}}^{\widetilde{\ell}_{u}}\\
&= C_{0}b_{1} \left(\frac{e^{(\theta + b_{1})\widetilde{\ell}_{d}} - e^{(\theta + b_{1})\widetilde{\ell}_{u}}}{\theta + b_{1}} - e^{b_{1} \widetilde{\ell}_{u}}\frac{e^{\theta \widetilde{\ell}_{d}} - e^{\theta \widetilde{\ell}_{u}}}{\theta} \right)1(b_{1} \neq 0) \\
& \qquad \qquad \qquad + C_{0}\frac{\theta e^{\theta \tilde{\ell}_{d}}\left(\widetilde{\ell}_{d} - \widetilde{\ell}_{u}\right) + e^{\theta \widetilde{\ell}_{u}} - e^{\theta \widetilde{\ell}_{d}}}{\theta^{2}(\widetilde{\ell}_{u} - \widetilde{\ell}_{d})} 1(b_{1} = 0). 
\end{align*}

We next compute $\varphi_{2,1}^{(n)}$ as $n \to \infty$. 
For the case $b_{1} \neq 0$, it follows from \eqn{psi 21} and  \eqn{d-1 eq} and \eqn{pi n 0 b1 neq 0} that
\begin{align*}
\lim_{n \to \infty} \varphi_{2,1}^{(n)}(\theta)  &= \lim_{n \to \infty} \psi_{2,1}^{(n)}\left(\frac{\theta}{\sqrt{n}} \right) \nonumber\\
& = \lim_{n \to \infty} \frac{\frac{(\rho_{1}^{(n)}e^{\frac{\theta}{\sqrt{n}}})^{\ell_{d}^{(n)}} - (\rho_{1}^{(n)} e^{\frac{\theta}{\sqrt{n}}})^{\ell_{u}^{(n)} + 1}}{1 - \rho_{1}^{(n)}e^{\frac{\theta}{\sqrt{n}}}} -  (\rho_{1}^{(n)})^{\ell_{u}^{(n)} + 1}\frac{e^{\frac{\theta}{\sqrt{n}} \ell_{d}^{(n)}} - e^{\frac{\theta}{\sqrt{n}}(\ell_{u}^{(n)} + 1)}}{1 - e^{\frac{\theta}{\sqrt{n}}}}}{1 - (\rho_{1}^{(n)})^{\ell_{u}^{(n)}-\ell_{d}^{(n)} + 2}}\pi^{(n)}(0,1) \nonumber\\
& = \frac{\frac{e^{(\theta + b_{1})\widetilde{\ell}_{d}} - e^{(\theta + b_{1})\widetilde{\ell}_{u}}}{\theta + b_{1}} - e^{b_{1} \widetilde{\ell}_{u}}\frac{e^{\theta \widetilde{\ell}_{d}} - e^{\theta \widetilde{\ell}_{u}}}{\theta}}{e^{b_{1}(\widetilde{\ell}_{u} - \widetilde{\ell}_{d})} - 1} \left(\lim_{n \to \infty} \sqrt{n} \pi^{(n)}(0,1) \right)\\
& = C_{0}b_{1} \left(\frac{e^{(\theta + b_{1})\widetilde{\ell}_{d}} - e^{(\theta + b_{1})\widetilde{\ell}_{u}}}{\theta + b_{1}} - e^{b_{1} \widetilde{\ell}_{u}}\frac{e^{\theta \widetilde{\ell}_{d}} - e^{\theta \widetilde{\ell}_{u}}}{\theta} \right),
\end{align*}
where we also use 
\begin{align*}
&1 - e^{\frac{\theta}{\sqrt{n}}} = -\frac{\theta}{\sqrt{n}} + o\left(\frac{1}{\sqrt{n}} \right), 
&&1 - \rho_{1}^{(n)} e^{\frac{\theta}{\sqrt{n}}} = -\frac{\theta + b_{1}}{\sqrt{n}} + o\left(\frac{1}{\sqrt{n}} \right), 
\end{align*}
as $n \to \infty$. 

For the case $\rho_{1}^{(n)} \neq 1$ and $b_{1} = 0$, 
\begin{align*}
&\left(\rho_{1}^{(n)}\right)^{\ell(n)}  = 1 + \ell(n)g_{1}(n) + \delta(n),
&e^{\frac{\theta}{\sqrt{n}}} = 1 + \frac{\theta}{\sqrt{n}} + g_{2}(n),
\end{align*}
where $\ell(n) \in \mathbb{N}$, $g_{1}(n) = o\left(\dfrac{1}{\sqrt{n}} \right)$ and $g_{2}(n) = o\left(\dfrac{1}{\sqrt{n}} \right)$ which may be different and $\delta(n)$ satisfies $\displaystyle{\lim_{n \to \infty} \frac{\delta(n)}{g_{1}(n)} } = 0$. In addition, 
\begin{align*}
1 - \rho_{1}^{(n)} e^{\frac{\theta}{\sqrt{n}}} &=  1 - (1 + g_{1}(n)+ \delta(n)) \left(1 + \frac{\theta}{\sqrt{n}} + g_{2}(n)\right)\\
&= -\frac{\theta}{\sqrt{n}} - g_{1}(n) -g_{2}(n) - \delta(n).
\end{align*}
Hence, 
\begin{align*}
&\varphi_{2,1}^{(n)}(\theta) \\
& \quad =\frac{\frac{(\rho_{1}^{(n)}e^{\frac{\theta}{\sqrt{n}}})^{\ell_{d}^{(n)}} - (\rho_{1}^{(n)} e^{\frac{\theta}{\sqrt{n}}})^{\ell_{u}^{(n)} + 1}}{1 - \rho_{1}^{(n)}e^{\frac{\theta}{\sqrt{n}}}} -  (\rho_{1}^{(n)})^{\ell_{u}^{(n)} + 1}\frac{e^{\frac{\theta}{\sqrt{n}} \ell_{d}^{(n)}} - e^{\frac{\theta}{\sqrt{n}}(\ell_{u}^{(n)} + 1)}}{1 - e^{\frac{\theta}{\sqrt{n}}}}}{1 - (\rho_{1}^{(n)})^{\ell_{u}^{(n)}-\ell_{d}^{(n)} + 2}}\pi^{(n)}(0,1)\\
& \quad =\frac{(\rho_{1}^{(n)}e^{\frac{\theta}{\sqrt{n}}})^{\ell_{d}^{(n)}}(1 - e^{\frac{\theta}{\sqrt{n}}})}{(1 - (\rho_{1}^{(n)})^{\ell_{u}^{(n)}-\ell_{d}^{(n)} + 2})(1 - \rho_{1}^{(n)}e^{\frac{\theta}{\sqrt{n}}})(1 - e^{\frac{\theta}{\sqrt{n}}})}\pi^{(n)}(0,1)\\
& \qquad - \frac{ (e^{\frac{\theta}{\sqrt{n}}})^{\ell_{d}^{(n)}}(\rho_{1}^{(n)})^{\ell_{u}^{(n)} + 1} (1 - \rho_{1}^{(n)} e^{\frac{\theta}{\sqrt{n}}}) }{(1 - (\rho_{1}^{(n)})^{\ell_{u}^{(n)}-\ell_{d}^{(n)} + 2})(1 - \rho_{1}^{(n)}e^{\frac{\theta}{\sqrt{n}}})(1 - e^{\frac{\theta}{\sqrt{n}}})}\pi^{(n)}(0,1)\\
& \qquad +\frac{  (\rho_{1}^{(n)} e^{\frac{\theta}{\sqrt{n}}})^{\ell_{u}^{(n)} + 1}e^{\frac{\theta}{\sqrt{n}}}(1-\rho_{1}^{(n)})}{(1 - (\rho_{1}^{(n)})^{\ell_{u}^{(n)}-\ell_{d}^{(n)} + 2})(1 - \rho_{1}^{(n)}e^{\frac{\theta}{\sqrt{n}}})(1 - e^{\frac{\theta}{\sqrt{n}}})}\pi^{(n)}(0,1)\\
& \quad =\frac{e^{\frac{\theta \ell_{d}^{(n)}}{\sqrt{n}}}(1 + \ell_{d}^{(n)} g_{1}(n) + \delta(n))\left(-\frac{\theta}{\sqrt{n}} - g_{2}(n) \right)\pi^{(n)}(0,1)  }{(- (\ell_{u}^{(n)}-\ell_{d}^{(n)} + 2)g_{1}(n) -\delta(n))\left(-\frac{\theta}{\sqrt{n}} - g_{1}(n) -g_{2}(n) - \delta(n) \right)\left(-\frac{\theta}{\sqrt{n}} - g_{2}(n) \right)}\\
& \qquad - \frac{e^{\frac{\theta{\ell_{d}^{(n)}}}{\sqrt{n}}}(1 +  (\ell_{u}^{(n)} + 1) g_{1}(n) + \delta(n)) \left(-\frac{\theta}{\sqrt{n}} - g_{1}(n) -g_{2}(n) - \delta(n) \right)\pi^{(n)}(0,1) }{(- (\ell_{u}^{(n)}-\ell_{d}^{(n)} + 2)g_{1}(n) -\delta(n))\left(-\frac{\theta}{\sqrt{n}} - g_{1}(n) -g_{2}(n) - \delta(n) \right)\left(-\frac{\theta}{\sqrt{n}} - g_{2}(n) \right)}\\
& \qquad + \frac{e^{\frac{\theta(\ell_{u}^{(n)} + 1)}{\sqrt{n}}}(1 + (\ell_{u}^{(n)} + 1 ) g_{1}(n) + \delta(n))\left(1 + \frac{\theta}{\sqrt{n}} + g_{2}(n) \right)(-g_{1}(n)-\delta(n))\pi^{(n)}(0,1)}{(- (\ell_{u}^{(n)}-\ell_{d}^{(n)} + 2)g_{1}(n) -\delta(n))\left(-\frac{\theta}{\sqrt{n}} - g_{1}(n) -g_{2}(n) - \delta(n) \right)\left(-\frac{\theta}{\sqrt{n}} - g_{2}(n) \right)}\\
&\quad =\frac{-g_{1}(n)( \theta e^{\theta \widetilde{\ell}_{d}}(\widetilde{\ell}_{d} - \widetilde{\ell}_{u}) + e^{\theta \widetilde{\ell}_{u}}- e^{\theta \widetilde{\ell}_{d}}) - \delta(n)}{-(\widetilde{\ell}_{u} - \widetilde{\ell}_{d})\theta^{2}g_{1}(n) - \delta(n)} \sqrt{n}\pi^{(n)}(0,1),
\end{align*}
as $n \to \infty$. Thus, by \eqn{pi n 0 b1 = 0}, we obtain
\begin{align}
\label{eqn:b1 = 0 rho1 neq 0}
\lim_{n \to \infty} \varphi_{2,1}^{(n)}(\theta) = C_{0}\frac{\theta e^{\theta \tilde{\ell}_{d}}\left(\widetilde{\ell}_{d} - \widetilde{\ell}_{u}\right) + e^{\theta \widetilde{\ell}_{u}} - e^{\theta \widetilde{\ell}_{d}}}{\theta^{2}(\widetilde{\ell}_{u} - \widetilde{\ell}_{d})}.
\end{align}
where, by \eqn{pi n 0 b1 = 0}, $\sqrt{n}\pi^{(n)}(0,1) \to C_{0}$ as $n \to \infty$ for the case $b_{1} = 0$.
 
Finally, we consider  the case $\rho_{1}^{(n)} = 1$ for any $n \ge n_{0}$ where $n_{0}$ is a certain integer. We then have
\begin{align*}
\varphi_{2,1}^{(n)}(\theta) &= \frac{(\ell_{u}^{(n)} - \ell_{d}^{(n)} + 1)e^{\frac{\theta}{\sqrt{n}} \ell_{d}^{(n)}} -\frac{e^{\frac{\theta}{\sqrt{n}}(\ell_{d}^{(n)} + 1)} - e^{\frac{\theta}{\sqrt{n}}(\ell_{u}^{(n)} + 2)}}{1 - e^{\frac{\theta}{\sqrt{n}}}}}{(\ell_{u}^{(n)} - \ell_{d}^{(n)} + 2)(1 -e^{\frac{\theta}{\sqrt{n}}})}\pi^{(n)}(0,1)\\
&= \frac{e^{\theta \widetilde{\ell}_{d}}(\widetilde{\ell}_{u} - \widetilde{\ell}_{d}) + \frac{e^{\widetilde{\ell}_{d}} - e^{\widetilde{\ell}_{u}}}{\theta}   + o\left(\frac{1}{\sqrt{n}}\right)}{-(\widetilde{\ell}_{u} - \widetilde{\ell}_{d})\theta + o\left(\frac{1}{\sqrt{n}} \right)} \sqrt{n}\pi^{(n)}(0,1), 
\end{align*}
and \eqn{b1 = 0 rho1 neq 0} is also obtained.

\end{document}